%% file: articulo6.tex
\documentclass[a4paper]{amsart}

\usepackage[english]{babel}
\usepackage[colorlinks]{hyperref}
\usepackage{multicol}
\usepackage{multirow}
\usepackage{graphicx}

\usepackage{amsmath}

\usepackage[latin1]{inputenc}
\usepackage{amsthm}
\usepackage{amsfonts}
\usepackage{enumitem}

\newtheorem{theorem} {Theorem}[section]

\newtheorem{lemma}[theorem]{Lemma}

\newtheorem{remark}[theorem]{Remark}

\newtheorem{definition}[theorem]{Definition}

\newcommand{\torre}[2]{\begin{tabular}{c} #1 \\ #2 \end{tabular}}

\newcommand{\conv}  {\operatorname{conv} }
\newcommand{\vol}  {\operatorname{vol} }
\newcommand{\Z}{\mathbb{Z}}
\newcommand{\N}{\mathbb{N}}
\newcommand{\R}{\mathbb{R}}

\begin{document}

\title{Lattice $3$-polytopes with six lattice points}
\author{M\'onica Blanco}
\author{Francisco Santos}
\thanks{Supported by grants MTM2011-22792 (both authors) and BES-2012-058920 (M.~Blanco) of the Spanish Ministry of Science.}
\maketitle

\setcounter{tocdepth}{1}
\tableofcontents
 
\begin{abstract}
In a previous paper we proved that for each $n$ there is only a finite number of (equivalence classes of) $3$-polytopes of lattice width larger than one and with exactly $n$ lattice points. We classified them for $n=5$ showing that there are exactly 9 different polytopes of width $2$, and none of larger width  (For $n=4$ it is a classical result of White that all empty tetrahedra have width one).

Here, we look at $n=6$. We show that there are $74$ polytopes of width $2$, two polytopes of width $3$, and none of larger width. We give explicit coordinates for representatives of each class, together with other invariants such as their \emph{oriented matroid} (or \emph{order type}) and volume vector. For example, according to the number of interior points these $76$ polytopes divide into $23$ tetrahedra with two interior points (\emph{clean} tetrahedra), $49$ polytopes with one interior point (the $49$ \emph{canonical} three-polytopes with five boundary points previously classified by Kasprzyk) and only $4$ \emph{hollow} polytopes (two tetrahedra, one quadrangular pyramid and one triangular bipyramid).

We also give a complete classification of three-polytopes of width one with $6$ lattice points. In terms of the oriented matroid of these six points, they lie in eight infinite classes and twelve individual polytopes.

Our motivation comes partly from the concept of \emph{distinct pair sum} (or dps) polytopes, which, in dimension $3$, can have at most $8$ lattice points. Among the $74+2$ classes mentioned above, exactly $44 + 1$ are dps. 
\end{abstract}
\bigskip


\newpage

\section{Introduction}

A \emph{lattice polytope} $P$ is the convex hull of a finite set of points in $\Z^d$ (or in a $d$-dimensional lattice).
$P$ is \emph{$d$-dimensional}, or a  \emph{$d$-polytope}, if it contains $d+1$ affinely independent points. 
 We call \emph{size} of $P$ its number $\#(P\cap \Z^d)$ of lattice points and \emph{volume} of $P$ its volume normalized to the lattice (that is, $d+1$ points form a simplex of volume one if and only if they are an affine lattice basis). More formally, 
\[
\vol(\conv\{p_1, \dots, p_{d+1}\}):=\left |\det \left( \begin{array}{ccc}
1 &\dots&1\\
p_1 &\dots&p_{d+1}\\
\end{array}
 \right)
 \right |
\]
with $p_i \in \Z^d$.

The \emph{width} of a lattice polytope is the minimum of $\max_{x\in P}f(x) -\min_{x\in P} f(x)$, over all possible (non-constant) choices of an integer linear functional $f:\Z^d\to \Z$. In particular, $P$ has \emph{width one} if its vertices lie in two parallel and consecutive lattice hyperplanes.

Two lattice polytopes $P$ and $Q$ are said \emph{$\Z$-equivalent} or \emph{unimodularly equivalent} if there is an affine integer unimodular transformation $t:\Z^d\to \Z^d$ with $t(P)=Q$. We call such a transformation a \emph{$\Z$-equivalence}. Volume, width, and size are obviously
invariant modulo $\Z$-equivalence.

\medskip

In dimension $2$, once we fix an $n\in \N$, there are finitely many $\Z$-equivalence classes of lattice $2$-polytopes of size $n$. In dimension $3$, in contrast, there are infinitely many classes for each size $n\ge 4$. Still, combining previous results it is easy to show that:

\begin{theorem}[\protect{\cite[Corollary 1.1]{5points}}]
\label{thm:finitewidth>1}
There are finitely many lattice $3$-polytopes of width greater than one for each size $n$.
\end{theorem}

So, it makes sense to classify separately, for each $n$, the $3$-polytopes of width one and those of  larger width.  Of width one there are infinitely many, but easy to describe: they consist of two $2$-polytopes of sizes $n_1$ and $n_2$ ($n_1+n_2=n$) placed on parallel consecutive planes (without loss of generality, the planes $z=0$ and $z=1$). 
For each of the two subconfigurations there is a finite number of possibilities, but infinitely many ways to ``rotate'' (in the integer sense, that is via an element of $SL(\Z,2)$) one with respect to the other. 
Of larger width there is none up to $n=4$ (all empty tetrahedra have width one, White~\cite{White}) and there are exactly 9 for $n=5$, all of width two~\cite{5points}. 
Here we completely classify $3$-polytopes of size $n=6$, showing that there are exactly 74 of width two, two of width three, and none of larger width (see precise results below).

\medskip

Our motivation comes partially from the notion of \emph{distinct pair-sum} lattice polytopes (or dps polytopes, for short). A lattice polytope $P$ is called \emph{dps} if all the pairwise sums $a+b$ ($a,b\in P\cap \Z^d$) are distinct. Equivalently, if $P\cap \Z^d$ contains neither three collinear points nor the vertices of a non degenerate parallelogram (\protect{\cite{ChoiLamReznick}}). They are also the lattice polytopes of Minkowski length one, in the sense of~\cite{Beckwith_etal}. 
Dps polytopes of dimension $d$ have size at most $2^d$, since they cannot have two lattice points in the same class modulo $(2\Z)^d$. (Dps $d$-polytopes of size $2^d$ are easy to construct~\cite{ChoiLamReznick}).

In dimension $2$ there are only two dps polytopes: a unimodular triangle, and a triangle of volume three with one interior point. The second one is maximal. 
Partial classification results for dps $3$-polytopes are contained in~\cite{Curcic}, but the following questions of
Reznick (\cite{Reznick-favorite}; see also~\cite[page~6]{ChoiLamReznick}) are open even in this case: 
\begin{itemize}
\item What is the range for the volume of dps polytopes of size $2^d$ in $\R^d$? 
\item Is every dps $d$-polytope a subset of one of size $2^d$?% (This was also asked in~\protect{\cite[page~6]{ChoiLamReznick}} for the case $d=3$).
\item How many  ``inequivalent''  dps polytopes of size $2^d$ are there in $\R^d$? 
\end{itemize}

To answer these  questions for $d=3$ it would be enough to continue the work in this paper to a complete classification of lattice $3$-polytopes of sizes $7$ and $8$. 
We do not know whether that is feasible (at least, it seems to us that new ideas and techniques would be needed) but it has to be noted that adding the dps constraint reduces the number of cases to study quite a bit: for example, Sections~\ref{sec:5coplanar},~\ref{sec:(2,2)coplanarity} and~\ref{sec:(2,1)coplanarity}, and parts of Sections~\ref{sec:width1} and~\ref{sec:(3,1)coplanarity}, of this paper could  be omitted if we only wanted to classify dps polytopes.

\medskip

Our main result can be summarized as follows:

\begin{theorem}
 There are exactly 76 $3$-polytopes of size $6$ and width $>1$. 74 of them have width $2$ and two have width $3$. $44$ and $1$ of those, respectively, are dps. 
\end{theorem}

Detailed information about each of these $76$ polytopes is given in Tables~\ref{table:6points-I} and \ref{table:6points-II}. This includes their oriented matroid, volume vector and width. 
Different sections in the tables correspond to the presence or not of certain coplanarities, as summarized in Table~\ref{table:result1} (and to different sections in the paper).
Also, explicit coordinates for a representative of each class are given in Section~\ref{sec:tables}.

\begin{table}[h]
\begin{center}
\begin{tabular}{|c|cc|c|c|}
\hline
\textbf{Description} & \textbf{$\#$ Polys.}&\textbf{ \small (dps)} & \textbf{$\#$ O.~m.'s} & \textbf{Where?}\\
\hline
\hline
$\exists\ 5$ coplanar points & $2$ %A.1, A.2
&{\small ($-$)}  & $2$ of $11$ %3.2 and 3.3 and NO (3.1, 4.1...4.5, 5.1...5.3)
&Section~\ref{sec:5coplanar}\\
\hline
\torre{$\exists\ (3,1)$ coplanarity,}{(no $5$ coplanar points)} & $20+1$%((B.1...B.6)*, B.7...B.10, (B.11...B.13)*, B.14*, B.15, C.1, C.2, C.4*, C.5*, C.6*) + (C.3)
 &{\small ($13$)}  &$13$ of $20$ %3.6, 3.8, 3.9, 3.11, 3.13, 4.13*, 4.17*, 4.18, 5.4*, 5.6*, 5.10*, 5.11*, 5.12* and NO(3.7, 4.6*, 4.7, 4.10, 4.12, 4.16, 4.19*)
 &Section~\ref{sec:(3,1)coplanarity} \\
\hline
\torre{$\exists\ (2,2)$ coplanarity}{(but none of  above)} & $4$ %D.1, D.2, E.1, E.2
&{\small ($-$)}  &$2$ of $15$ %5.5 and 5.13 and NO(3.4, 3.5, 3.10, 3.12, 4.8, 4.14, 4.20, 5.7, 5.8, 5.9, 5.14, 5.15)
 &Section~\ref{sec:(2,2)coplanarity}  \\
\hline
\torre{$\exists\ (2,1)$ coplanarity}{(but none of  above)} & $17$ %F
 &{\small ($-$)}  &$3$ of $5$ %4.11, 4.21, 4.22 and NO(2.1, 4.9, 4.15) 
 &Section~\ref{sec:(2,1)coplanarity}  \\
\hline
\torre{No coplanarity,}{$1$ interior point} & $20$ %G
 &{\small ($20$)}  &$1$ of $1$  %6.2
 &Section~\ref{sec:2*(4,1)-1} \\
\hline
\torre{No coplanarity,}{$2$ interior points}  & $11+1$ %H + H.12
&{\small ($11+1$)}  & $1$ of $1$ %6.1
 &Section~\ref{sec:2*(4,1)-2} \\
\hline
\torre{No coplanarity,}{no interior points}  & $0$ &{\small ($0$)}  & $0$ of $2$ &Theorem~\ref{thm:Howe} \\
\hline
\end{tabular}
\caption{Number of lattice $3$-polytopes of size $6$ and width $>1$ according to coplanarities present in them. Where it says ``$x+1$" it means that ``$x$ have width two and one has width three".}
\label{table:result1}
\end{center}
\end{table}

Emptiness in the last line of Table~\ref{table:result1} follows from the following result that Scarf attributes to Howe:

\begin{theorem}[\protect{\cite[Thm.~1.3]{Scarf}}]
\label{thm:Howe}
If all lattice points of a lattice $3$-polytope are vertices then it has width $1$.
%In particular, all maximal $3$-polytopes with that property consist of two empty parallelograms in consecutive parallel lattice planes.
\end{theorem}

The previous to last column of Table~\ref{table:result1}, and the first column of Tables~\ref{table:6points-I}--\ref{table:6points-II}, contains oriented matroid information. This is used both as book-keeping and as a way of to extract useful information from the tables. For this reason we have worked out in Section~\ref{subsec:om} the full classification of oriented matroids of six points in $\R^3$, of which there are $55$. Only $22$ of them are realized by the six lattice points of some lattice polytope of size $6$ and width $>1$. This is indicated in Table~\ref{table:result1} as follows:  we say we have ``$i$ of $j$'' oriented matroids  in a particular row meaning that there are $j$ oriented matroids of six points in $\R^3$ corresponding to the description in that row, but only $i$ of them arise for the six points of a lattice $3$-polytope of size $6$ and width $>1$.

Table~\ref{table:om} in Section~\ref{subsec:om} classifies the $55$ oriented matroids according to the presence or not of the coplanarities of Table~\ref{table:result1}, but also according to their number of vertices and interior points. The latter allows us to easily translate the information from Tables~\ref{table:6points-I}--\ref{table:6points-II} into a classification of lattice polytopes of size six and width $>1$ according to their number of vertices and interior lattice points. This is summarized in Table~\ref{table:result2}. 
Here we call a polytope \emph{clean} if all its boundary lattice points are vertices, \emph{hollow} if it has no interior lattice points, \emph{canonical} if it has exactly one interior point (usually assumed to be the origin) and \emph{terminal} if it is canonical and clean. This usage seems quite established~\cite{KasprzykNill,NillZiegler,Reznick-clean}.
Polytopes with five vertices can have a triangular bipyramid or a quadrangular pyramid as convex hull.
Polytopes with six vertices do not appear in the classification since, by Theorem~\ref{thm:Howe}, they have width one. 

\begin{table}[h]
\begin{center}
\begin{tabular}{|c|c|c|cc|c|}
\hline
\textbf{$\#$vert.} &{\textbf{\torre{$\#$int.}{pts.}}} & \textbf{Description}& \textbf{$\#$ polys.} & \textbf{ \small (dps)} & \textbf{$\#$ O.~m.'s}\\
\hline
\hline
\multirow{3}{*}{$4$} & $2$& clean tetrahedra& $22+1$ %B.2*, B.5*, B.6*, I.1*, B.11*, F.1--F.6, (H.1--H.11)* + H.12*
&{\small ($16+1$)}  & $4$ of $5$\\ %4.17*, 5.10*, 4.21, 6.1* and NO(5.9)
\cline{2-6}
 &$1$ &canonical tetrahedra & $10+1$ %B.7, B.1*, C.2, C.4*, C.5*, F.13--F.17 + C.3
 &{\small ($3$)}  & $5$ of $6$\\ % 3.8, 4.13*, 3.11, 5.4*, 4.11 and NO(4.12)
\cline{2-6}
 &$0$ &hollow tetrahedra& $2$ %A.1, C.1
 &{\small ($-$)}  &\,\ $2$ of $10$\\ %3.2, 3.6 and NO(3.1, 4.1, 4.2, 4.4, 5.1, 3.7, 4.6*, 2.1)
\hline
\multirow{4}{*}{$5$} &\multirow{2}{*}{$1$}& terminal quad.~pyramid & $3$ %B.9, E.1, E.2
&{\small ($-$)} & $2$ of $3$\\ %3.9, 5.5 and NO(4.10)
& & terminal tri.~bipyramid & $35$ %B.3*, B.4*, B.8, B.10, B.12*, B.13*, B.15, D.1, D.2, F.7--F.12, (G.1--G.20)*
&{\small ($24$)} & $7$ of $8$\\ %5.12*, 4.18, 3.13, 5.11*, 5.13, 4.22, 6.2* and NO(4.19*)
\cline{2-6}
 &\multirow{2}{*}{$0$}& hollow quad.~pyramid & $1$ %A.2
 &{\small ($-$)} & $1$ of $6$\\ %3.3 and NO(4.3, 4.5, 5.2, 4.7, 3.5)
& & hollow tri.~bipyramid & $1$ %C.6*
&{\small ($1$)} & $1$ of $5$\\ %5.6* and NO(4.16, 3.10, 4.9, 4.15)
\hline
6 & 0 & clean \& hollow  & $0$ &{\small ($-$)} &\,\ $0$ of $12$\\ %NO(...)
\hline
\end{tabular}
\caption{Classification of lattice $3$-polytopes of size six and width $>1$ according to their number of vertices and interior points. As in the previous table, ``$x+1$'' means ``$x$ of width two and one of width three''.}
\label{table:result2}
\end{center}
\vspace{-.5cm}
\end{table}

Let us remark that clean tetrahedra and canonical $3$-polytopes were previously classified, even for numbers of lattice points much greater than six:
\begin{itemize}
\item Kasprzyk has classified all canonical $3$-polytopes (\cite{Kasprzyk}), regardless of their number of lattice points. (That they are finitely many is an instance of Hensley's Theorem (\cite[Thm.~4.3]{Hensley}): for each $k\ge 1$ and fixed dimension $d$ there is a bound on the volume of lattice $d$-polytopes with $k$ interior lattice points). There are $674,688$ of them, with their number of boundary lattice points going up to $38$ and their volume up to $72$. The list is published as a searchable database (\cite{Kasprzyk-database}) in which we have checked that, indeed, there are 49 canonical $3$-polytopes with five boundary points, $38$ of them terminal, in agreement with our classification.
\item For clean tetrahedra with $k$ interior points Pikhurko improved Hensley's Theorem to a volume upper bound of $84.63 k$ (\cite{Pikhurko}). Using this, Curcic classified clean tetrahedra with up to $35$ interior points (\cite{Curcic}). 
In this range the maximum volume that arises is $12k+8$, which had been previously conjectured by Duong for all $k$ (\cite{Duong}). 
For $k=2$ Curcic and Duong (personal communication) found exactly the same $23$ clean tetrahedra with two interior points that we have found.
\end{itemize}

For completeness, in Section~\ref{sec:width1} we also work out a complete classification of lattice $3$-polytopes of size six and width one. There are infinitely many, but  all the infinite series lie in eight particular oriented matroids. There are another twelve individual polytopes, all with different oriented matroid. See Remark~\ref{rem:width1-redundancy}.
\bigskip

\textbf{Acknowledgment:} We thank Bruce Reznick and Han Duong for useful comments and references related to our and their work.

\input{Preliminaries6}

\input{6puntos}

\bigskip

\input{6width1}

\medskip

\input{5coplanar}

\bigskip

\input{3,1-coplanarity}

\bigskip

\input{2,2-coplanarity}

\bigskip

\input{2,1-coplanarity}

\bigskip

\input{gp-intpoint}

\bigskip

\input{tables}

\end{document}

%% file: Preliminaries6.tex
%%!TEX root =articulo6.tex
\section{Preliminaries}
\subsection{Volume vectors}
Since $\Z$-equivalence preserves volume, the following \emph{volume vector} is invariant under it:

\begin{definition}
\label{def:VolumeVectors}
Let $A=\{  p_1,p_2, \dots ,p_n  \}$, with $n \ge d+1$, be a set of lattice points in $\Z^d$.
The \emph{volume vector of $A$} is the vector
\[
w= (w_{i_1\dots i_{d+1}})_{1 \le i_1 <\dots <i_{d+1}\le n} \in \mathbb{Z}^{ {n \choose d+1}}
\]
where
\begin{equation}
w_{i_1\dots i_{d+1}}=w_{\{i_1,\dots, i_{d+1}\}}:=\det \left( \begin{array}{ccc}
	1	&\dots	&1\\      
	p_{i_1}	&\dots	&p_{i_{d+1}}\\
\end{array}
\right).
\label{eq:volume}
\end{equation}
\end{definition}

The definition implicitly assumes a specific ordering of the $n$ points in $A$. For six points $\{p_1,\dots,p_6\}$ we always order the entries of the volume vector lexicographically as:
\begin{eqnarray*}
w&=&(w_{1234},w_{1235},w_{1236},w_{1245},w_{1246},w_{1256},w_{1345}, \\
&&w_{1346},w_{1356},w_{1456},w_{2345},w_{2346},w_{2356},w_{2456},w_{3456})
\end{eqnarray*}

\begin{remark}
\label{rmk:(d+2)-volumeVector}
The volume vector encodes, among other things, the unique (modulo a scalar factor) dependence among each set of $d+2$ points $\{p_{1}, \dots ,p_{{d+2}}\}$ that affinely span $\R^d$. This dependence is as follows, where $I_k=\{1, \dots , {d+2}\} \setminus \{k\}$:
\begin{eqnarray*}
 \sum_{k=1}^{d+2} (-1)^{k-1} \cdot w_{I_k} \cdot p_{k}&=&0 \\
  \sum_{k=1}^{d+2}  (-1)^{k-1} \cdot w_{I_k} &=&0
 \end{eqnarray*}
\end{remark}

The volume vector is often, but not always, a \emph{complete} invariant for $\Z$-equivalence:
%We now look at the converse question: if two point sets of the same size have the same volume vector, are they necessarily $\Z$-equivalent? The answer is \emph{almost} yes. 

\begin{theorem}[\protect{\cite[Theorem 2.3]{5points}}]
\label{thm:VolumeVectors}
Let $A=\{ p_1, \dots , p_n  \}$ and $B= \{ q_1, \dots , q_n \}$ be $d$-dimensional subsets of $\Z^d$ and suppose they have the same volume vector $(w_I)_{I \in \binom{[n]}{d+1}}$ with respect to the given ordering. Then:

\begin{enumerate}
\item There is a unique affine map $t: \R^d\to \R^d$ with $t(p_i)=q_i$ for all $i$, and it has $\det(t)=1$.

\item If $\gcd_{I \in \binom{[n]}{d+1}} \ (w_I)=1$, then $t$ has integer coefficients, so it is a $\Z$-equivalence between $P$ and $Q$.

\end{enumerate}
\end{theorem}
\medskip

\subsection{The oriented matroid of six points in $\R^3$}
\label{subsec:om}

The oriented matroid of $n$ points in $\R^d$ can be defined as the vector of signs in the volume vector. Hence, it is a $\Z$-equivalence invariant. In this section we give the full list of oriented matroids for six points in $\R^3$.
We assume a certain familiarity with oriented matroid theory at the level covered, for example, in \cite[Chapter 6]{Ziegler} or \cite[Chapter 4]{deLoeraRambauSantos2010}.

What we are looking at are oriented matroids of \emph{rank four} with \emph{six elements}. But not all of them can arise from an affine point configuration. The additional conditions that we need are:
\begin{enumerate}
\item The oriented matroid must be \emph{acyclic}. Put differently, it does not have positive circuits.
\item The oriented matroid must not have any \emph{parallel elements} (circuits of signature $(1,1)$).
\end{enumerate}

To classify these oriented matroids we look at their duals: the oriented matroids with six elements and rank two that are \emph{totally cyclic} and do not have \emph{cocircuits} of type $(1,1)$. These are all representable, so we can picture them as vector configurations in $\R^2$. With an exhaustive procedure we compute the following full list. (We do not claim originality. Rank four o.~m.'s have been  classified up to 10 elements (\cite{Matsumoto-et-al})):
\begin{theorem}
\label{thm:OM}
There are $55$ acyclic oriented matroids of rank four with six elements and without parallel elements. Their duals are represented in Figure~\ref{fig:OM}.
\end{theorem}

\begin{figure}[htb]
\centering
\includegraphics[scale=0.52]{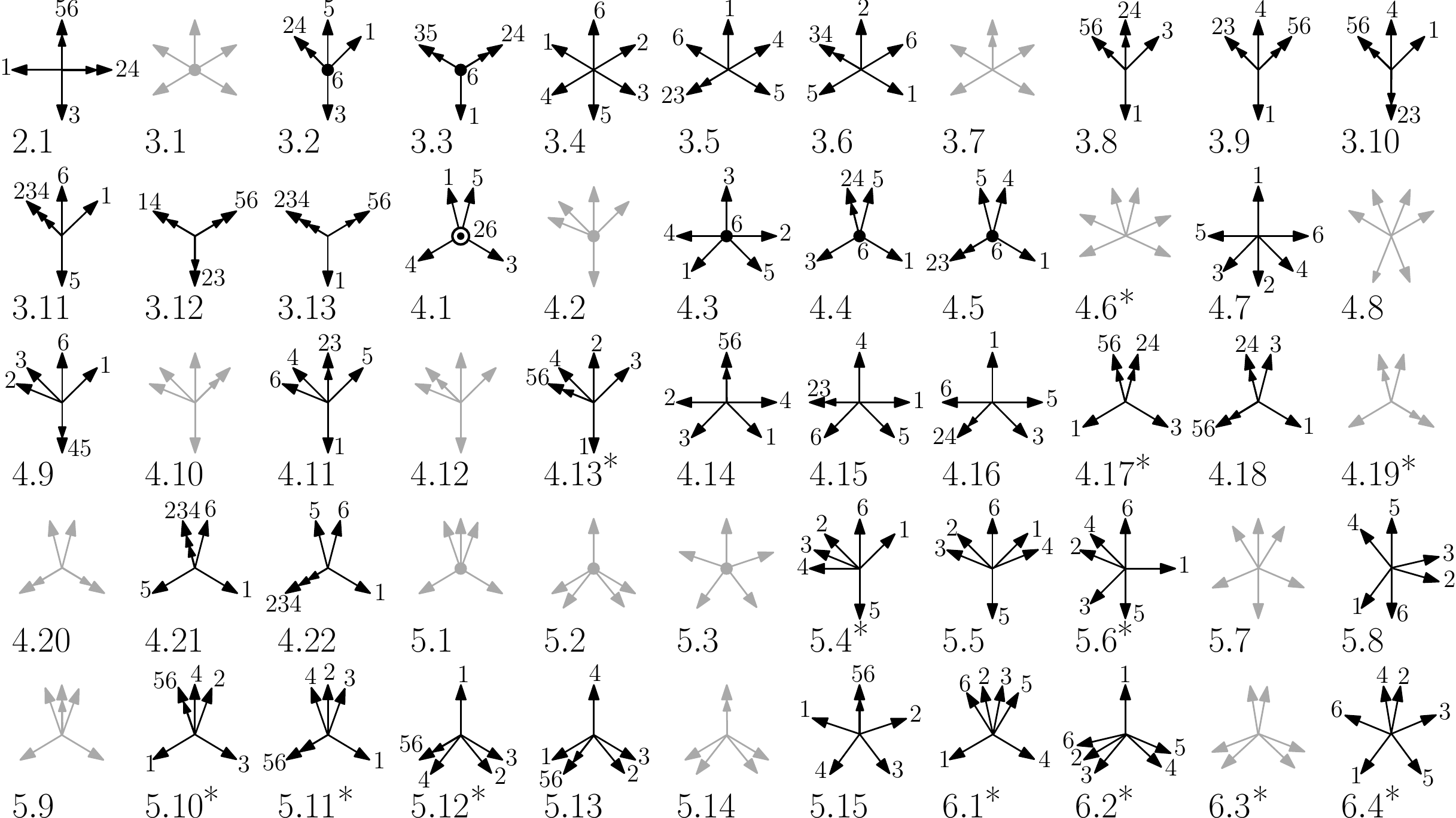}
\caption{The $55$ possible oriented matroids of six distinct points in $\R^3$ (represented by their duals, of rank two).% Those non-realizable with lattice polytopes of size $6$ are in gray. Those that are dps are marked with *.
}
\label{fig:OM}
\end{figure}

%Observe that some of the duals do have circuits of type $(1,1)$ (graphically shown as multiple arrows along a vector) or even $(1,0)$ (dots at the origin). This is not in contradiction to what we say above.

\begin{table}[h]
\small
\centering
\begin{tabular}{|c|c|c|c|c|c|c|c|}
\hline
\textbf{\torre{$\#$ Vertices} {$\#$ Int. pts.}} & 
\torre{$v=4$}{$i=2$}&
\torre{$v=4$}{$i=1$}&
\torre{$v=4$}{$i=0$}&
\torre{$v=5$}{$i=1$}&
\torre{$v=5$}{$i=0$}&
\torre{$v=6$}{$i=0$}\\
\hline
\hline
\begin{tabular}{c}$\exists\ 5$\\ coplanar\\ points \end{tabular}
 & & &
\begin{tabular}{c} 3.1,  3.2\\  4.1, 4.2\\ 4.4, 5.1 \end{tabular} &
&
\begin{tabular}{c} 3.3\\ 4.3\\ 4.5\\ 5.2\end{tabular} &
5.3\\
\hline
\begin{tabular}{c} $\exists\ (3,1)$\\ coplanarity\\but no $5$\\ coplanar\\ points \end{tabular} &
\torre{4.17*}{5.10*} & 
\begin{tabular}{c} 3.8\\3.11\\4.12\\4.13*\\5.4*\end{tabular} &
\begin{tabular}{c} 3.6\\ 3.7\\ 4.6* \end{tabular} &
\begin{tabular}{c} 3.9\\ 3.13 \\4.10\\4.18\\4.19*\\5.11*\\5.12*  \end{tabular} &
\begin{tabular}{c} 4.7\\ 4.16\\ 5.6*\end{tabular} &
\\
\hline
\begin{tabular}{c} $\exists\ (2,2)$\\ coplanarity\\but none of\\  the above \end{tabular} &
5.9 &
&
&
\begin{tabular}{c} 5.5\\ 5.13  \end{tabular} &
\torre{3.5}{3.10}&
\begin{tabular}{c} 3.4, 3.12\\ 4.8,  4.14\\ 4.20, 5.7\\ 5.8, 5.14, 5.15 \end{tabular}
\\
\hline
\begin{tabular}{c} $\exists\ (2,1)$\\ coplanarity\\but none of\\the above \end{tabular} &
4.21 &
4.11 &
2.1  &
4.22 &
\torre{4.9}{4.15} &
\\
\hline
\torre{no}{coplanarity} & 6.1* & & & 6.2* & & \torre{6.3*}{6.4*} \\
\hline
\end{tabular}
\caption{The $55$ oriented matroids of six points in $\R^3$ classified according to number of vertices, interior points, and presence of certain coplanarities.}
\label{table:om}
\end{table}

To get this list we have classified the oriented matroids of rank 2 according to their number of cocircuits (i.e., of different lines containing vectors of the configuration). In fact, each oriented matroid is labeled $N.M$ where $N$ indicates its number of cocircuits and $M$ is just an indexing label. For example, the four oriented matroids labeled $6.M$ are the \emph{uniform} ones, whose duals correspond to point configurations in \emph{general position}.
Once the number of cocircuits (i.e., lines) is fixed we just need to look at all the possible ways of putting one or more vectors along each line (or at the origin) keeping the two constraints that we have for our oriented matroids, which translate to:
\begin{itemize}
\item 
The oriented matroid must be totally cyclic (that is, \emph{the vectors must positively span $\R^2$}) and 
\item there must not be a cocircuit of type $(1,1)$ (\emph{no line can contain four or more vectors}).
\end{itemize}

A posteriori, it turns out that many of the oriented matroids in the classification cannot be realized by the six lattice points in a $3$-polytope of size six. Those appear in gray in the figure. 
Also, since being dps depends only on the oriented matroid (it corresponds to the non-existence of a $(2,2)$ or a $(2,1)$ circuit) the figure and tables indicates the oriented matroids corresponding to dps polytopes with an asterisk.

For future reference, the six elements in each oriented matroid are labeled with the numbers $1$ to $6$. The volume vectors in tables~\ref{table:6points-I} and \ref{table:6points-II} are given with respect to this order.

\subsection{Polytopes with $4$ or $5$ lattice points}
\label{subsec:5points}
\label{subsec:4points}

Three-dimensional lattice polytopes of size four, \emph{empty tetrahedra}, were classified 50 years ago:

\begin{theorem}[Classification of empty tetrahedra, White~\cite{White}]
\label{thm:empty_tetrahedra}
Every empty lattice tetrahedron is $\Z$-equivalent to the following $T(p,q)$, for some $q\in \N$ and $p\in\{0,q-1\}$ with $\gcd(p,q)=1$:
\[
T(p,q)= \conv \{ (0,0,0), (1,0,0), (0,0,1), (p,q,1) \}
\]
Moreover, $q=|\vol \, T(p,q)|$ and $T(p,q)$ is equivalent to $T(p',q)$ if and only if $p'=\pm p^{\pm 1} \pmod q$.
\end{theorem}

%If the tetrahedron is unimodular, i.e., if $q=1$, then the choice of representative is usually $T_0$.

%\begin{lemma}[\cite{5points}]
%\label{lemma:LemmaA}
%The lattice tetrahedron $T=\conv\{(0,0,0),$ $(1,0,0),$ $(0,1,0), (a,b,q)\}$ is empty (with respect to the integer lattice $\Z^3$) if, and only if, at least one of the following happens:
%\begin{enumerate}
%\item[(i)] $a \equiv 1 \pmod q$ and $\gcd(b,q)=1$.
%\item[(ii)] $b \equiv 1 \pmod q$ and $\gcd(a,q)=1$.
%\item[(iii)] $a + b \equiv 0 \pmod q$ and $\gcd(a,q)=1$.
%\end{enumerate}
%\end{lemma}

While working out our classification, we often need to check whether certain tetrahedra are empty and to find their type $(p,q)$. We have implemented a MATLAB program to do this for us. Our input is a  lattice tetrahedron given by its vertices $T=\{p_1,p_2,p_3,p_4\}$. By Theorem~\ref{thm:empty_tetrahedra}, $T$ is empty if and only if:
\begin{itemize}

\item All its edges are primitive segments (the coordinates of each edge-vector are relatively prime), and

\item $T$ has width one with respect to a functional that is constant on two opposite edges (in Theorem~\ref{thm:empty_tetrahedra} this would be the functional $z$). There are three choices for this pair of edges, and computing the (primitive) functional that is constant on a particular choice is straightforward.
%. Since in empty tetrahedra the minimum width is always achieved with respect to a pair of edges, we check the width with respect to all three pairs and the tetrahedron has width one if and only if (at least) one of these three widths is one. 
\end{itemize}

Three-dimensional lattice polytopes of size five were fully classified in~\cite{5points}. Remember
that five  points $A=\{p_1,p_2,p_3,p_4,p_5\}$ affinely spanning  $\R^3$ have a unique affine dependence. The \emph{Radon partition} of $A$ is obtained by looking at the signs of coefficients in this dependence. 
We say that the point configuration $A$ of size five  has \emph{signature $(i,j)$} if this dependence has $i$ positive and $j$ negative coefficients. 
The five possibilities for $(i,j)$ are $(2,1)$, $(2,2)$, $(3,2)$, $(3,1)$ and $(4,1)$. (Observe that 
$(i,j)$ and $(j,i)$ are the same signature).

Table~\ref{table:5points} shows the full classification of three-polytopes of size five.
The polytopes are grouped according to their signatures and we include a representative for each equivalence class, as well as the volume vector of the class. However, 
in order for the volume vector to highlight signature, and taking into account Remark~\ref{rmk:(d+2)-volumeVector}, for a $5$-point configuration we 
modify our conventions and write volume vectors in the form
\[
w=(w_{2345}, \ -w_{1345},\ w_{1245},\ -w_{1235},\ w_{1234})
\]
where $w_{ijkl}$ is as in Equation~\ref{eq:volume}.
With this choice, the signature of $A$ is just the number of positive and negative entries in the volume vector, and the sum of coordinates in the volume vector vanishes.

\begin{table}[htb]
\footnotesize
\centering
\begin{tabular}{|c|c|c|c|}
\hline
\textbf{Sign.} &\textbf{Volume vector} & \textbf{Width} & \textbf{Representative} \\
\hline
$(2,2)$&  $(-1,1,1,-1,0)$ &$1$& $(0,0,0)$, $(1,0,0)$, $(0,1,0)$, $(1,1,0)$,$(0,0,1)$\\
  \hline
$(2,1)$&  \begin{tabular}{c}
    $(-2q,q,0,q,0)$  \\
    $0\le p\le \frac{q}{2}$,\\ $\gcd(p,q)=1$ 
  \end{tabular}
  & $1$  & $(0,0,0)$, $(1,0,0)$, $(0,0,1)$, $(-1,0,0)$,$(p,q,1)$\\
  \hline
$(3,2)$&  \begin{tabular}{c}
    $(-a-b,a,b,1,-1)$ \\
    $0<a\le b$, \\$\gcd(a,b)=1$
  \end{tabular}
  & $1$ & $(0,0,0)$, $(1,0,0)$, $(0,1,0)$, $(0,0,1)$,$(a,b,1)$\\
  \hline
$(3,1)$&  \begin{tabular}{c}
    $(-3,1,1,1,0)$  \\
    $(-9,3,3,3,0)$
  \end{tabular} &   \begin{tabular}{c}
    $1$  \\
    $2$
  \end{tabular}  &\begin{tabular}{l}
$(0,0,0)$, $(1,0,0)$, $(0,1,0)$, $(-1,-1,0)$,$(0,0,1)$\\
$(0,0,0)$, $(1,0,0)$, $(0,1,0)$, $(-1,-1,0)$,$(1,2,3)$
  \end{tabular}\\
  \hline
&  $(-4,1,1,1,1)$ & $2$ & $(0,0,0)$, $(1,0,0)$, $(0,0,1)$, $(1,1,1)$,$(-2,-1,-2)$\\
& $(-5,1,1,1,2)$ & $2$ & $(0,0,0)$, $(1,0,0)$, $(0,0,1)$, $(1,2,1)$,$(-1,-1,-1)$\\
& $(-7,1,1,2,3)$ & $2$ & $(0,0,0)$, $(1,0,0)$, $(0,0,1)$, $(1,3,1)$,$(-1,-2,-1)$\\
$(4,1)$&  $(-11,1,3,2,5)$ &$2$ & $(0,0,0)$, $(1,0,0)$, $(0,0,1)$, $(2,5,1)$,$(-1,-2,-1)$\\
&  $(-13,3,4,1,5)$ &  $2$ & $(0,0,0)$, $(1,0,0)$, $(0,0,1)$, $(2,5,1)$,$(-1,-1,-1)$\\
&  $(-17,3,5,2,7)$ &  $2$ & $(0,0,0)$, $(1,0,0)$, $(0,0,1)$, $(2,7,1)$,$(-1,-2,-1)$\\
&  $(-19,5,4,3,7)$ &  $2$ & $(0,0,0)$, $(1,0,0)$, $(0,0,1)$, $(3,7,1)$,$(-2,-3,-1)$\\
&  $(-20,5,5,5,5)$ &  $2$ & $(0,0,0)$, $(1,0,0)$, $(0,0,1)$, $(2,5,1)$,$(-3,-5,-2)$\\
\hline
\end{tabular}
\caption{Complete classification of lattice $3$-polytopes of size $5$. Those of signatures $(3,2)$, $(3,1)$ and $(4,1)$ are dps.}
\label{table:5points}
\end{table}

\begin{remark}
Let $A=\{p_1,p_2,p_3,p_4,p_5\} \subset \Z^3$. By Theorem~\ref{thm:VolumeVectors}, for $\conv(A)$ to have size $5$ it is sufficient that $A$ has volume vector equal to one of those listed in Table~\ref{table:5points}, except in the following cases:

\begin{itemize}
\item Signature  $(2,1)$ and  $q\ge 2$.

\item Signature  $(3,1)$ and  volume vector $(-9,3,3,3,0)$.

\item Signature  $(4,1)$ and  volume vector $(-20,5,5,5,5)$.
\end{itemize}

Whenever we are in these three exceptions, we need to check that the transformation between $\conv(A)$ and the representative element of the class 
(the map mentioned in part 1 of Theorem~\ref{thm:VolumeVectors}) is indeed integer. In the first two cases, the following two lemmas are very useful for doing this:

\begin{lemma}[\protect{\cite[Thm. 3.4]{5points}}]
\label{lemma:(3,1)thm}
A configuration $A$ of size five and with volume vector $(-9,3,3,3,0)$ has no additional lattice points in $\conv(A)$ if, and only if, in coordinates for which the $(3,1)$ circuit is $\conv \{o,e_1,e_2,-e_1-e_2\}$, the fifth point $(a,b,q)$ verifies $a \equiv -b \equiv \pm 1 \pmod 3$.
\end{lemma}

\begin{lemma}[\protect{\cite[Thm. 3.2(3)]{5points}}]
\label{thm:lemma(2,1)}
A configuration $A$ of size five and with volume vector $(-2q,q,q,0,0)$ has no additional lattice points in $\conv(A)$ if, and only if, in coordinates for which the $(2,1)$ circuit is $\conv \{o,e_2,-e_2\}$ and the fourth point is $e_1$, the fifth point $(a,b,q)$ verifies $a \equiv 1 \pmod q$ and $\gcd(b,q)=1$.
\end{lemma}

\end{remark}

%% file: 6puntos.tex
%%!TEX root =articulo6.tex

\section{Overview of the classification scheme}
\label{sec:overview}

In what follows, we  denote by $A=\{p_1,\dots ,p_6\}\subset \Z^3$ a configuration of six lattice points, and consider the polytope  $P=\conv(A)\subset \R^3$, which may or not have size $6$. 
We denote $A^i:=A \setminus \{p_i\}$ and $P^i$ (or $P \setminus \{p_i\}$) the polytope $\conv (A^i)$. It is quite straightforward that if $P$ has size $6$, then for every vertex $p_i$ of $P$,  $P^i$ is a lattice polytope of size $5$.
Conversely, if for every vertex $p_i$ of $P$ we have that $P^i$ has size five, then $P$ has size six.

Remember that a configuration of $5$ lattice points has a unique circuit, consisting of $a+b \le 5$ points ($a$ positive and $b$ negative ones) and that we say that such configuration (or its corresponding polytope) has \emph{signature $(a,b)$}. 

Although our main interest is in polytopes of width greater than one, 
Section~\ref{sec:width1} classifies polytopes of size six and width one. We look separately at configurations with
 $5+1$, $4+2$ or $3+3$ points in the two consecutive parallel planes. Perhaps the most interesting point in the classification is that one of the two uniform oriented matroids with six vertices and without interior points (the one labelled 6.3) can not arise for the six lattice points of a lattice polytope. That is, no lattice $3$-polytope of size six and without coplanarities has an octahedron as its convex hull (Lemma~\ref{lemma:NO-6.3}).
\smallskip

For width greater than one, the classification involves a quite long case study, developed in Sections~\ref{sec:5coplanar}-\ref{sec:gp-intpoint}. For this, suppose that $\conv(A)$ has size six and width greater than one, then one of the following things occurs:

\begin{itemize}[leftmargin=.8cm]

\item $A$ contains $5$ coplanar points (Case A, Section~\ref{sec:5coplanar}). Then $A$ consists of one of the six polygons of size five (see Figure~\ref{fig:2dim-5-overview}) 
plus an extra point at lattice distance at least two from it. It is not hard to show that only configurations 1 and 2 in the figure allow this sixth point to be placed adding no additional lattice points in $\conv(A)$, and they allow it in a single way modulo $\Z$-equivalence.

\begin{figure}[htb]
\vspace{-2mm}
\centering
\includegraphics[scale=0.6]{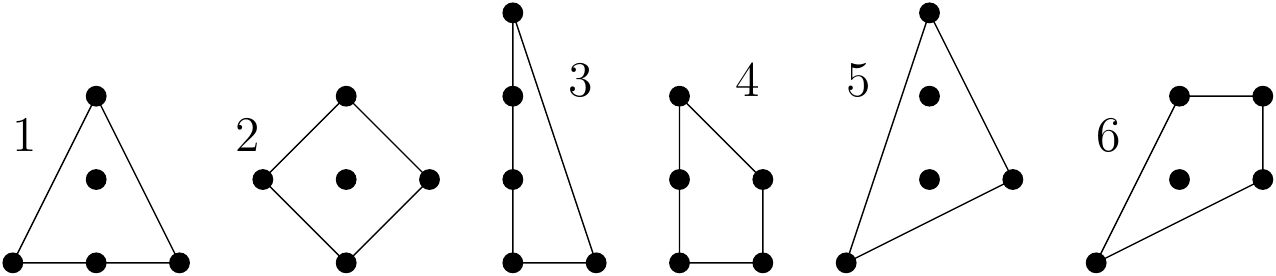}
\caption{The lattice polygons of size five}
\label{fig:2dim-5-overview}
\end{figure}

\item $A$ contains a coplanarity of type $(3,1)$ in a certain plane $H$ (and no five coplanar points). Sections~\ref{sec:(3,1)coplanarity1} (Case B) and~\ref{section:3-1-sameside} (Case C) treat this case, depending on whether the  other two points of $A$ lie in opposite or the same side of $H$. If they lie in opposite sides then they are both at distance $1$ or $3$ of $H$, by the classification of size five. If they lie on the same side, then either they are both vertices of $\conv(A)$ (and lie at distance $1$ or $3$) or one is an interior point (at distance $1$ or $3$ from $H$) and the other is a vertex. In the cases where both are guaranteed to be at distance $1$ or $3$ we use what we call the \emph{parallel-planes method} and when one is interior we use the \emph{$(4,1)$-extension method} (see details below).
\medskip

\item $A$ contains a coplanarity of type $(2,2)$ (and none of the above). This is treated in Sections~\ref{sec:(2,2)coplanarity1} (Case D) and~\ref{sec:(2,2)+(4,1)} (Case E) in much the same way as the $(3,1)$ case, except things are now simpler because every time we said ``distance $1$ or $3$'' in the previous paragraph we can now say ``distance $1$''.
\medskip

\item All coplanarities in $A$ come from $(2,1)$ collinearities (Case F, treated in Section~\ref{sec:(2,1)coplanarity}). We first show that the $(2,1)$ collinearity must be unique (two of them would produce either width one or $5$ coplanar points) and that removing from $A$ one or the other extremal collinear points we  get a configuration of signature $(4,1)$ and size $5$ (Lemma~\ref{lemma:2-1-strong}). Once we have this we can use the \emph{$(4,1)$-extension method}.
\medskip

\item $A$ is in general position (no coplanarities). $A$ must have interior points, by Theorem~\ref{thm:Howe}.
This case is treated in Sections~\ref{sec:2*(4,1)-1} and~\ref{sec:2*(4,1)-2} depending on whether $A$ has one (Case G) or two (Case H) interior points, which corresponds exactly to oriented matroids 6.2 and 6.1, respectively.

It turns out that both oriented matroids have the following useful property: there are two vertices $p_i$ and $p_j$ of $A$ (the elements labeled 5 and 6 in Figure~\ref{fig:OM}) such that both $A^i$ and $A^j$ have signature $(4,1)$. Configurations like these of size $6$ can all be obtained \emph{gluing} two configurations of size $5$ and signature $(4,1)$ along four points. There are only eight configurations of signature $(4,1)$ and size $5$ that we need to consider (see Table~\ref{table:5points}) and (at most; this count contains some repetition) $\left(\binom{8}{2}+8\right)\times 4^2 \times 4!$ possible ways to glue them, which we have checked one by one with the aid of a computer.
\end{itemize}
\medskip

Let us explain the parallel-planes and $(4,1)$-extension methods mentioned above:
We use the \emph{parallel-planes method} when we can guarantee that $A$ is contained in three parallel planes $H_1$, $H_2$ and $H_3$ and we know (or pose without loss of generality) the coordinates of all points but one. In this case we look at what conditions must the coordinates of the last point satisfy for  $\conv(A)$ not to have extra lattice points in the intermediate plane $H_2$. This is a $2$-dimensional problem that can be solved graphically. Observe that $\conv(A)\cap H_2$ is
the convex hull of the union of $A\cap H_2$ and $A_2\cap H_2$, where $A_2$ denotes the set of intersection points $p_ip_j\cap H_2$ of the segments $\{p_ip_j : p_i\in H_1, p_j\in H_3\}$ with $H_2$. This idea sounds more complicated than it is, because we usually use it with $|A\cap H_1|=|A\cap H_3|=1$, so $A_2$ is a single point. The parallel planes method gives us a finite (and small) list of possible positions for the unknown point, which we then check one by one because they could produce extra lattice points in other parallel planes between $H_1$ and $H_3$ (we do not assume $H_1$, $H_2$ and $H_3$ to be consecutive).

We use the \emph{$(4,1)$-extension method} when we know that there is a vertex $p$ in $A$ such that $A\setminus \{p\}$ has signature $(4,1)$ and, moreover, we know an expression of $p$ as an affine combination of the other points. The latter happens when $p$ is part of a coplanarity, because all coplanarities in $A$ satisfy one of the following three affine relations (unless $A$ has five coplanar points):
\[
3p_i = p_j + p_k + p_l, \qquad
p_i+p_j = p_k + p_l, \qquad
2p_i = p_j + p_k.
\]
If we know $A$ to be in these conditions we simply need to go through the $8\times 4!$ possible ways to map $A\setminus \{p\}$ to one of the eight configurations of signature $(4,1)$ from Table~\ref{table:5points}, compute the corresponding $p$, and check whether the convex hull of the result has size six. This is done computationally.

To check that a particular $A$ has convex hull of size six we typically set up a triangulation $T$ of $A$ (see~\cite{deLoeraRambauSantos2010}) and check emptiness of each tetrahedron in $T$ as explained in Section~\ref{subsec:4points}.

%% file: 6width1.tex
%%!TEX root =articulo6.tex

\section{Polytopes of size $6$ and width $1$}
\label{sec:width1}

Polytopes of width one have all their lattice points in two consecutive lattice planes, which we assume to be $z=0$ and $z=1$. We denote $P_0$ and $P_1$ the corresponding $2$-dimensional configurations in each plane. There is a finite number of possibilities for these two configurations, but (as long as both $P_0$ and $P_1$ have at least two points) infinitely many ways to rotate one with respect to the other. 

There are three possibilities depending on the sizes of $P_0$ and $P_1$.

\subsubsection*{$\bullet$ $|P_0|=5$, $|P_1|=1$}
$P_0$ is one of the six $2$-dimensional polytopes of size $5$ displayed in the first column of Table~\ref{table:width1-(5,1)}. The equivalence class of $P$ depends only on $P_0$, and not on the choice of the sixth point at $z=1$, which can be assumed to be $(0,0,1)$. Hence we get exactly six  types of polytopes of this kind. None of them is dps, since $5$ points in the same plane cannot be dps.

\begin{table}[htbp]
\centering
\footnotesize
\begin{tabular}{|cc|c|}
\hline
        \multicolumn{2}{|c|}{\rule{0pt}{3ex} \textbf{$P_0=P \cap \{z=0\}$}} & \textbf{O.~M.}\\
\hline
\raisebox{-0.45\height}{\vspace{2 mm} \includegraphics[scale=0.4]{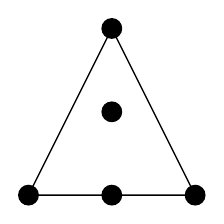}}&
$\begin{pmatrix}
    0 & 1 & 0 &-1 & 0\\
    0 & 0 & 1 & 0 & 2\\ 
    0 & 0 & 0 & 0 & 0
\end{pmatrix}$
 & 3.2\\
\hline
\raisebox{-0.45\height}{\includegraphics[scale=0.4]{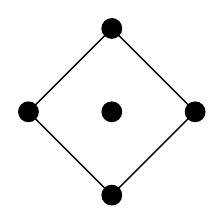}}&
$\begin{pmatrix}
    0 & 1 & 0 &-1 & 0\\
    0 & 0 & 1 & 0 & -1\\ 
    0 & 0 & 0 & 0 & 0 
\end{pmatrix}$
 & 3.3\\
\hline
\raisebox{-0.45\height}{\includegraphics[scale=0.4]{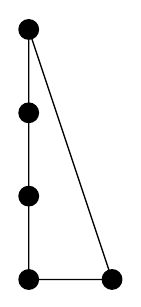}}&
$\begin{pmatrix}
    0 & 1 & 0 & 0 & 0\\
    0 & 0 & 1 & 2 & 3\\ 
    0 & 0 & 0 & 0 & 0
\end{pmatrix}$
 & 4.1\\
\hline
\end{tabular}
\ 
\begin{tabular}{|cc|c|}
\hline
        \multicolumn{2}{|c|}{\rule{0pt}{3ex} \textbf{$P_0=P \cap \{z=0\}$}} & \textbf{O.~M.}\\
\hline
\raisebox{-0.45\height}{\includegraphics[scale=0.4]{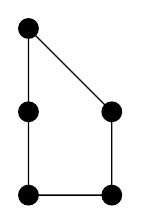}}&
$\begin{pmatrix}
    0 & 1 & 0 & 1 & 0\\
    0 & 0 & 1 & 1 & 2\\ 
    0 & 0 & 0 & 0 & 0
\end{pmatrix}$
 & 4.3\\
\hline
\raisebox{-0.45\height}{\includegraphics[scale=0.4]{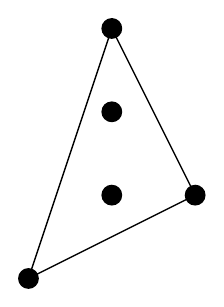}}&
$\begin{pmatrix}
    0 & 1 & 0 &-1 & 0\\
    0 & 0 & 1 &-1 & 2\\
    0 & 0 & 0 & 0 & 0
\end{pmatrix}$
 & 4.4\\
\hline
\rule{0pt}{6ex}\raisebox{-0.45\height}{\includegraphics[scale=0.4]{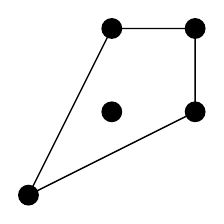}}&
$\begin{pmatrix}
    0 & 1 & 0 &-1 & 1\\
    0 & 0 & 1 &-1 & 1\\
    0 & 0 & 0 & 0 & 0
\end{pmatrix}$
 & 4.5\\
\hline
\end{tabular}
\caption{Polytopes of width one with $5+1$ lattice points.}
\label{table:width1-(5,1)}
\end{table}

\subsubsection*{$\bullet$ $|P_0|=4$, $|P_1|=2$}
Either $P_0$ is one of the three $2$-dimensional polytopes of size $4$, or it consists of four collinear points, as 
displayed in the first column of Table~\ref{table:width1-(4,2)}.
$P_1$ is a primitive segment. Since the classification depends solely on the relative position of the segment $P_1$ (modulo translation) with respect to $P_0$, we assume $P_1$ to have vertices $(0,0,1)$ and $(a,b,1)$, with $a\ge0$. We look at all possibilities one by one in  
Table~\ref{table:width1-(4,2)}. Symmetries of $P_0$ are often used to reduce the possibilities for $a$ and $b$.

\begin{table}[htbp]
\centering
\footnotesize
\begin{tabular}{|c|cc|c|}
\hline
\rule{0pt}{3ex}     \textbf{$P_0=P \cap \{z=0\}$}             & \multicolumn{2}{c|}{\textbf{$P_1=P \cap \{z=1\}$}} & \textbf{O.~M.} \\ 
\hline
\multirow{3}{*}{\torre{\includegraphics[scale=0.4]{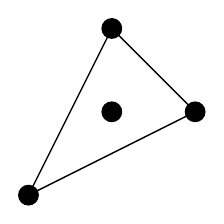}}{$\begin{pmatrix}
    0 & 1 & 0 &-1\\
    0 & 0 & 1 &-1\\
    0 & 0 & 0 & 0
\end{pmatrix}$} } & \raisebox{-0.45\height}{\includegraphics[scale=0.4]{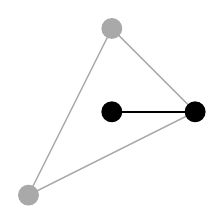}}&
$\begin{pmatrix}
    0 & 1\\
    0 & 0\\
    1 & 1
\end{pmatrix}$  & 4.16 \\ 
\cline{2-4} 
\rule{0pt}{6ex}                      & \raisebox{-0.45\height}{\includegraphics[scale=0.4]{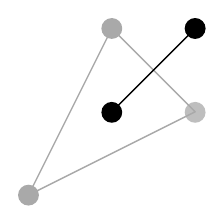}}&
$\begin{pmatrix}
    0 & 1\\
    0 & 1\\
    1 & 1
\end{pmatrix}$  & 4.7 \\ 
\cline{2-4} 
\rule{0pt}{6ex}                  & \raisebox{-0.45\height}{\includegraphics[scale=0.4]{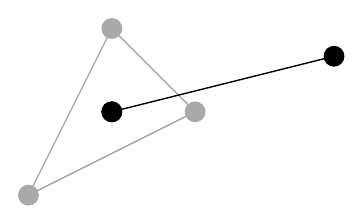}}&
$\begin{pmatrix}
    0 & a\\
    0 & b\\
    1 & 1
\end{pmatrix}$ \torre{\torre{$0<b<a$}{$\gcd(a,b)=1$}}{$2b \neq a$}  & 5.6* \\ 
\hline
\rule{0pt}{6ex} \multirow{3}{*}{\torre{\includegraphics[scale=0.4]{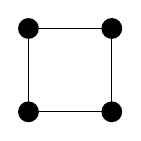}}{$\begin{pmatrix}
    0 & 1 & 0 &1\\
    0 & 0 & 1 &1\\
    0 & 0 & 0 & 0
\end{pmatrix}$} } &  \raisebox{-0.45\height}{\includegraphics[scale=0.4]{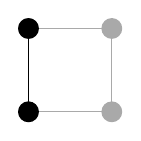}}&
$\begin{pmatrix}
    0 & 0\\
    0 & 1\\
    1 & 1
\end{pmatrix}$ & 3.4 \rule{0pt}{6ex}\\ 
\cline{2-4} 
\rule{0pt}{6ex}                  & \raisebox{-0.45\height}{\includegraphics[scale=0.4]{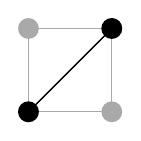}}&
$\begin{pmatrix}
    0 & 1\\
    0 & 1\\
    1 & 1
\end{pmatrix}$& 4.14 \rule{0pt}{6ex}\\ 
\cline{2-4} 
\rule{0pt}{6ex}                  & \raisebox{-0.45\height}{\includegraphics[scale=0.4]{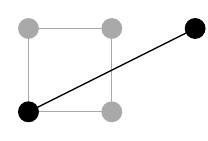}}&
$\begin{pmatrix}
    0 & a\\
    0 & b\\
    1 & 1
\end{pmatrix}$ \torre{$0<b<a$}{$\gcd(a,b)=1$} & 5.8 \\ 
\hline
\multirow{5}{*}{\torre{\begin{tabular}{c} \null \\ \null \\ \null\\ \null\end{tabular}}{\torre{\includegraphics[scale=0.4]{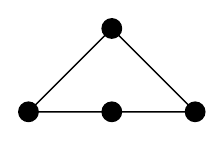}}{$\begin{pmatrix}
    0 & 1 & 0 &-1\\
    0 & 0 & 1 &0\\
    0 & 0 & 0 & 0
\end{pmatrix}$} }} & \raisebox{-0.45\height}{\includegraphics[scale=0.4]{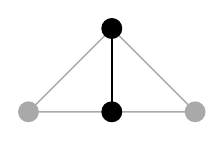}}&
$\begin{pmatrix}
    0 & 0\\
    0 & 1\\
    1 & 1
\end{pmatrix}$& 3.10 \rule{0pt}{6ex}\\ 
\cline{2-4} 
\rule{0pt}{6ex}                  & \raisebox{-0.45\height}{\includegraphics[scale=0.4]{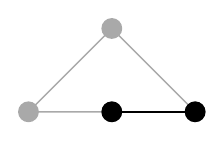}}&
$\begin{pmatrix}
    0 & 1\\
    0 & 0\\
    1 & 1
\end{pmatrix}$& 4.3 \rule{0pt}{6ex}\\ 
\cline{2-4} 
\rule{0pt}{6ex}                  & \raisebox{-0.45\height}{\includegraphics[scale=0.4]{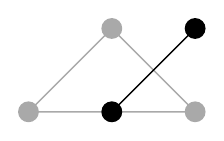}}&
$\begin{pmatrix}
    0 & 1\\
    0 & 1\\
    1 & 1
\end{pmatrix}$& 3.5 \rule{0pt}{6ex}\\ 
\cline{2-4} 
\rule{0pt}{6ex}                  & \raisebox{-0.45\height}{\includegraphics[scale=0.4]{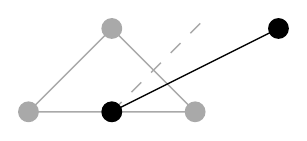}}&
$\begin{pmatrix}
    0 & a\\
    0 & b\\
    1 & 1
\end{pmatrix}$ \torre{$0<b<a$}{$\gcd(a,b)=1$}  & 4.15 \rule{0pt}{6ex}\\ 
\cline{2-4} 
\rule{0pt}{6ex}                  & \raisebox{-0.45\height}{\includegraphics[scale=0.4]{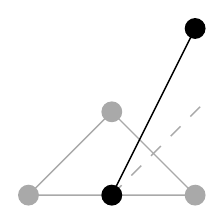}}&
$\begin{pmatrix}
    0 & b\\
    0 & a\\
    1 & 1
\end{pmatrix}$ \torre{$0<b<a$}{$\gcd(a,b)=1$}& 4.9 \rule{0pt}{6ex}\\ 
\hline
\rule{0pt}{6ex} \raisebox{-0.45\height}{\includegraphics[scale=0.4]{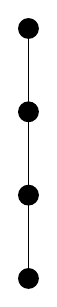}} \quad $\begin{pmatrix}
    0 & 0 & 0 & 0\\
    0 & 1 & 2 & 3\\
    0 & 0 & 0 & 0
\end{pmatrix}$    & \raisebox{-0.45\height}{\includegraphics[scale=0.4]{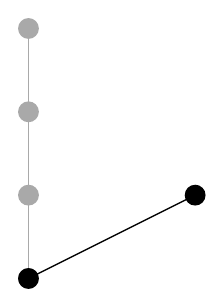}}&
$\begin{pmatrix}
    0 & a\\
    0 & b\\
    1 & 1
\end{pmatrix}$ \torre{$0\le b<a$}{$\gcd(a,b)=1$}& 4.1\\ 
\hline
\end{tabular}
\caption{Polytopes of width one with $4+2$ lattice points.}
\label{table:width1-(4,2)}
\end{table}

\subsubsection*{$\bullet$ $|P_0|=3$, $|P_1|=3$}
Either both $P_0$ and $P_1$ are unimodular triangles, or both are segments with three collinear points, or we have one of each. These three cases split in subcases, studied one by one in Table~\ref{table:width1-(3,3)}.
The rest of this section is devoted to the case when both $P_0$ and $P_1$ are  triangles and with no coplanarities, which deserves special attention.

\begin{table}[htbp]
\centering
\footnotesize
\begin{tabular}{|c|cc|c|}
\hline
  \rule{0pt}{3ex}     
  \textbf{$P_0=P \cap \{z=0\}$}             & \multicolumn{2}{c|}{\textbf{$P_1=P \cap \{z=1\}$}} & \textbf{O.~M.} \\ 
\hline
 \rule{0pt}{6ex} \begin{tabular}{c} \includegraphics[scale=0.4]{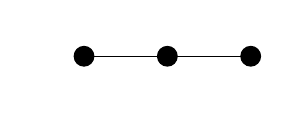}\\ $\begin{pmatrix}
    0 & 1 &-1\\
    0 & 0 & 0\\
    0 & 0 & 0
\end{pmatrix}$  \end{tabular}             & \raisebox{-0.45\height}{\includegraphics[scale=0.4]{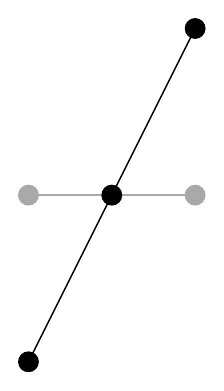}}&
$\begin{pmatrix}
    0 & b &-b\\
    0 & a & -a\\
    1 & 1& 1
\end{pmatrix}$ \torre{$0 \le b <a$}{$\gcd(a,b)=1$} &  2.1\\ 
\hline
  \multirow{8}{*}{
      \begin{tabular}{c} 
        \null \\ \null \\ \null\\ \null \\ \null \\ \null \\ \null \\ \null \\ \null\\
        \includegraphics[scale=0.4]{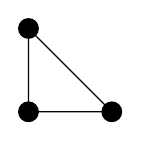}\\
        {$\left( 
          \begin{array}{ccc}
            0 & 1 & 0\\ 
            0 & 0 & 1\\
            0 & 0 & 0
          \end{array}
        \right)$} 
      \end{tabular}
    }
\rule{0pt}{6ex} & \raisebox{-0.45\height}{\includegraphics[scale=0.4]{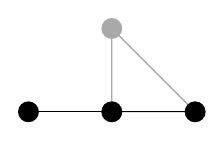}}&
$\begin{pmatrix}
    0 & 1 &-1\\
    0 & 0 & 0\\
    1 & 1& 1
\end{pmatrix}$ &  4.3\\ 
\cline{2-4} 
\rule{0pt}{6ex}                  &\raisebox{-0.45\height}{\includegraphics[scale=0.4]{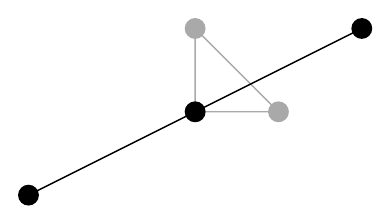}}&
$\begin{pmatrix}
    0 & a &-a\\
    0 & b & -b\\
    1 & 1& 1
\end{pmatrix}$ \torre{$0<b\le a$}{$\gcd(a,b)=1$} &  4.15\\ 
\cline{2-4} 
\rule{0pt}{6ex}                  & \raisebox{-0.45\height}{\includegraphics[scale=0.4]{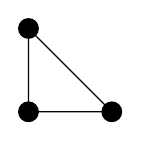}}&
$\begin{pmatrix}
    0 & 1 & 0\\
    0 & 0 & 1\\
    1 & 1& 1
\end{pmatrix}$ &  3.4\\ 
\cline{2-4} 
\rule{0pt}{6ex}                  & \raisebox{-0.45\height}{\includegraphics[scale=0.4]{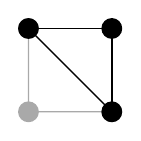}}&
$\begin{pmatrix}
    1 & 0 & 1\\
    0 & 1 & 1\\
    1 & 1& 1
\end{pmatrix}$ & 3.12 \\ 
\cline{2-4} 
\rule{0pt}{6ex}                  & \raisebox{-0.45\height}{\includegraphics[scale=0.4]{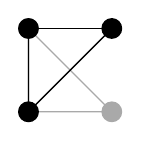}}&
$\begin{pmatrix}
    0 & 0 &1\\
    0 & 1 & 1\\
    1 & 1& 1
\end{pmatrix}$ & 4.14\\ 
\cline{2-4} 
\rule{0pt}{6ex}                  &  \raisebox{-0.45\height}{\includegraphics[scale=0.4]{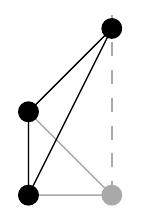}}&
$\begin{pmatrix}
    0 & 0 &1\\
    0 & 1 & a\\
    1 & 1& 1
\end{pmatrix}$ $a>1$ &  5.8\\ 
\cline{2-4} 
\rule{0pt}{6ex}                  &  \raisebox{-0.45\height}{\includegraphics[scale=0.4]{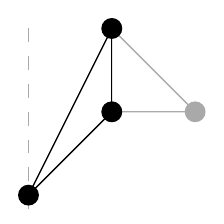}}&
$\begin{pmatrix}
    0 & 0 &-1\\
    0 & 1 & a\\
    1 & 1& 1
\end{pmatrix}$ $a>3$ & 5.15\\ 
\cline{2-4} 
\rule{0pt}{6ex}                  & \raisebox{-0.45\height}{\includegraphics[scale=0.4]{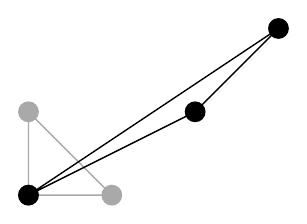}}&
$\begin{pmatrix}
    0 & a & c\\
    0 & b & d\\
    1 & 1& 1
\end{pmatrix}$ \begin{tabular}{c}$ad-bc=\pm1$ \\ $a,b,c,d>0$\\ $c+d>a+b$\\ \end{tabular} &  6.4*\\ 
\hline
\end{tabular}
\caption{Polytopes of width one with $3+3$ lattice points.}
\label{table:width1-(3,3)}
\end{table}

In the table, we first look at subcases where these triangles have edges parallel to one another, which produces exactly five different oriented matroids. We then look at what happens when no parallel edges are present. In this case the configuration is in general and convex position, so its oriented matroid must be one of 6.3 and 6.4. The following lemma proves that 6.3 can never arise:

\begin{lemma}
\label{lemma:NO-6.3}
Let $P$ be a lattice $3$-polytope of size $6$, consisting of two unimodular triangles in consecutive lattice planes, with no two parallel edges. Then its  oriented matroid is $6.4$ of Figure~\ref{fig:OM}.
\end{lemma}

\begin{proof}
Let $T=\conv\{p_1,p_2,p_3\}$ be a triangle with primitive edges in the plane $z=0$. For $i=1,2,3$, let $t_i$ be the outward normal vector of the edge $p_jp_k$, where $\{i,j,k\}=\{1,2,3\}$, normalized to be primitive. (Each $t_i$ is the  $90$ degree rotation of the corresponding edge vector). Clearly for each pair in $i,j \in\{1,2,3\}$, $|\det(t_i,t_j)|$ equals the normalized volume of $T$.

Let now $S=\conv\{q_i\}$ be another triangle with primitive edges in the parallel plane $z=1$, with no edges parallel to those in $T$. Let $s_i$ be the corresponding normal vectors as before. We now have six distinct normal vectors.

We claim that if the oriented matroid of $P=\conv(T \cup S)$ was $6.3$, then the sequence of the six normal vectors, cyclically ordered according to their angle, 
should alternate one vector from each triangle. For this, the following figure shows the (dual of) oriented matroid $6.3$ and the Schlegel diagram of its realization (an octahedron):

\centerline{\includegraphics[width=0.35\textwidth]{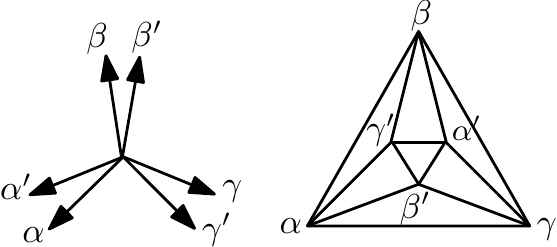}}

The six triangles in the region between $\alpha\beta\gamma$ and $\alpha'\beta'\gamma'$ in the Schlegel diagram correspond to the six normal vectors. The fact that the triangles alternate between using two vertices from $S$ and from $T$ implies that the normal vectors alternate (to see this, observe that two points $p_i$ and $p_j$ form a triangle with a point $q_k$ if and only if the normal vector of $p_ip_j$ belongs to the normal cone of $q_k$).

Suppose now that our triangles are unimodular and let us see that such an alternation of normal vectors is impossible.
We can now assume without loss of generality that the normal vectors for $T$ are $t_1=(1,0)$, $t_2=(0,1)$ and $t_3=(-1,-1)$ and that each $s_i$ is between $t_j$ and $t_k$ where $\{i,j,k\}=\{1,2,3\}$.
Then $s_3=(c,d)$ with $c,d >0$ and, by symmetry, we assume that $s_2=(a,-b)$ with $a> 0$, $b>0$ (if $s_2=-t_2$, a coplanarity arises). 
In order for $S$ to be unimodular we would need $\det(s_2,s_3)= \pm 1$, but $\det(s_2,s_3)=ad+bc \ge 2$ since $a,b,c,d>0$.
\end{proof}

Once we know the oriented matroid to be  $6.4$, and that the sequence of the normal vectors (see the proof above) cannot alternate between $T$ and $S$,  this sequence must contain two consecutive $s$'s and two consecutive $t$'s. The next figure shows the dual oriented matroid and its Schlegel diagram.

\centerline{
\includegraphics[width=0.35\textwidth]{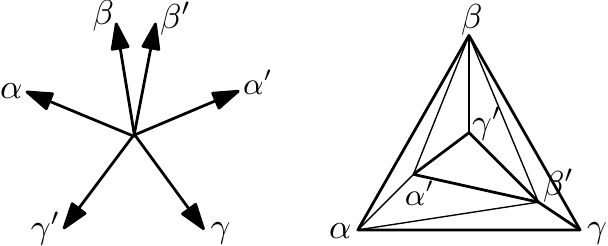}
}

In terms of the vertex cones of the triangles this implies that the vertex cone of $\gamma$ in $T$ contains two edge directions of $S$, the vertex cone in $\alpha$ contains one, and the vertex cone in $\beta$ contains none. Let us see that this implies the description in the table for the space of parameters $a,b,c,d$ to be complete.

Without loss of generality, $T$ has vertices $p_1=(0,0,0)$, $p_2=(1,0,0)$ and $p_3=(0,1,0)$, and 
$S$ has vertices $p_4=q_1=(0,0,1)$, $p_5=q_2=(a,b,1)$ and $p_6=q_3=(c,d,1)$, with $ad-bc=1$ for $S$ to be unimodular. We assume
two edge vectors of $S$ contained in the vertex cone of $p_1$, so that $a,b,c,d >0$. And we assume the vertex cone in $p_3$ to contain the remaining edge vector, which gives $c+d >a+b$.
%\[\bordermatrix{ & 1 & 2 & 3 & 4 & 5 &6 \cr 
%    & 0    & 1    & 0    & 0    &a   &  c\cr
%    & 0    & 0    & 1    & 0    &b   & d\cr
%    & 0    & 0    & 0    & 1    &1   &  1\cr }\]
%
%

\begin{remark}
By Theorem~\ref{thm:Howe}, all lattice $3$-polytopes with no lattice points other than its vertices must have width $1$.
This is an oriented matroid property, satisfied by the following oriented matroids from Figure~\ref{fig:OM}: 3.4, 3.12, 4.8, 4.14, 4.20. 5.3, 5.7, 5.8, 5.14, 5.15, 6.3 and 6.4. In particular, all the lattice $3$-polytopes of size six corresponding to any of those oriented matroids are included in the classification done in this section. This implies that the following six oriented matroids are not realizable by the six points of a lattice $3$-polytope of size six: 4.8, 4.20, 5.3, 5.7, 5.14, and 6.3. (The fact that 6.3 cannot be realized is the content of Lemma~\ref{lemma:NO-6.3}).
\end{remark}

\begin{remark}
\label{rem:width1-redundancy}
The three possibilities for partitioning the six points into two consecutive lattice planes
are not mutually exclusive, since a configuration may have width 1 with respect to two or more functionals. In particular, Tables~\ref{table:width1-(5,1)}-\ref{table:width1-(3,3)} contain some redundancy. However, the oriented matroid allows us to easily detect them:

\begin{itemize}
\item The oriented matroids 3.4, 4.3 and 4.14 appear in two or three of the tables. All appearances of each correspond to one and the same lattice polytope.

\item The oriented matroid 4.1 appears in Tables~\ref{table:width1-(5,1)} (as a single polytope) and~\ref{table:width1-(3,3)} (as an infinite family). The former is the special case $a=1$ of the latter.

\item The oriented matroid 5.8 appears in Tables~\ref{table:width1-(4,2)} and~\ref{table:width1-(3,3)}. Both appearances produce infinite families, but the latter (depending on a single parameter) is the special case $b=1$ of the former.

\item The oriented matroid 4.15 appears in Tables~\ref{table:width1-(4,2)} and~\ref{table:width1-(3,3)}. Both appearances produce infinite families. These two families are not only different but actually disjoint: in Table~\ref{table:width1-(3,3)} the three collinear points span a segment that is parallel to the triangle spanned by the other three, in Table~\ref{table:width1-(4,2)} the same segment and triangle are never parallel.
\end{itemize}

That is: after removing redundancies, all $3$-polytopes of size six and width one fall into nine infinite series (two of them for the same oriented matroid 4.15) plus twelve individual polytopes.
\end{remark}

%% file: 5coplanar.tex
%%!TEX root =articulo6.tex

\section{Case A: Polytopes with $5$ coplanar points}
\label{sec:5coplanar}

If  $P$ has five points $p_1,\dots,p_5$ in a lattice plane (for example $z=0$) and the sixth point $p_6$ is outside this plane, the first five points form one of the six $2$-dimensional configurations displayed in Figure~\ref{fig:2dim-5}. 
 
\begin{figure}[htb]
\begin{center}
\includegraphics[scale=0.6]{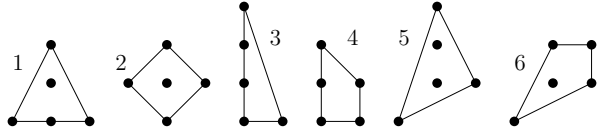}
\caption{The 2-dimensional polygons with five points}
\label{fig:2dim-5}
\end{center}
\end{figure}

In cases 3, 4 and 6 it is easy to conclude that $P$ has width one: 
\begin{itemize}
\item In cases 4 and 6, the five coplanar points contain the vertices of a quadrilateral, say $p_1p_2p_3p_4$. The classification of polytopes of size $5$ implies that in order for $P^5=\conv\{p_1,p_2,p_3,p_4,p_6\}$ not to have extra lattice points we need it to have width one with respect to the quadrilateral facet, so the same happens for $P$.

\item In case 3 let $p_1$, $p_2$, $p_3$ and $p_4$ be the four collinear points. The classification of polytopes of size $5$ implies that in order for $P^4=\conv\{p_1,p_2,p_3,p_5,p_6\}$ not to have extra lattice points we need it to have width one with respect to the edges $p_1p_2p_3$ and $p_5p_6$, so the same happens for $P$.
\end{itemize}

In cases 1, 2 and 5 we consider the following  coordinates for the six points:
\[
p_1=(0,0,0), \ \ p_2=(1,c,0),\ \  p_3=(0,1,0), \ \ p_4=(-1,d,0), \ \ p_5=(0,2,0), \ \ p_6=(a,b,q).
\]
Here $c$ and $d$ depend solely on which case we are in, $a$ and $b$ can be considered modulo $q$, and we assume $q>1$ since $q=1$ implies width one.

In all cases, both $P^2=\conv\{p_1,p_3,p_4,p_5,p_6\}$ and $P^4=\conv\{p_1,p_2,p_3,p_5,p_6\}$ are in the conditions of Lemma~\ref{thm:lemma(2,1)} (under a certain change of coordinates). Moreover, the lemma implies that in order for $P^4$ and $P^2$ not to have extra lattice points we need, respectively, $a \equiv 1 \pmod q$ and $a \equiv -1 \pmod q$. That is, we can assume $q=2$ and $a=1$. Once we know this, there are only two possibilities for $p_6$, namely $p_6=(1,0,2)$ and $p_6=(1,1,2)$, which we discuss case by case:

\begin{itemize}
\item \textbf{Case 1:} $c=d=0$. If $p_6=(1,0,2)$ then the midpoint of $p_2p_6$ is integer. Thus the only possibility is $p_6=(1,1,2)$, which
gives configuration A.1 of Table~\ref{table:6points-I}.

\item \textbf{Case 2:} $c=d=1$. If $p_6=(1,1,2)$ then the midpoint of $p_2p_6$ is integer. Thus the only possibility is $p_6=(1,0,2)$, which
produces configuration A.2 of Table~\ref{table:6points-I}.

\item \textbf{Case 5:} $c=0$ and $d=-1$, so both possibilities create a new lattice point at $z=1$; the midpoint of $p_4p_6$ for $p_6=(1,1,2)$, and 
the midpoint of $p_2p_6$ for $p_6=(1,0,2)$.
\end{itemize}
\medskip

In summary:

\begin{theorem}
\label{thm:6points-1}
Among the lattice $3$-polytopes of size six with $5$ coplanar points, there are exactly $2$ equivalence classes of width two, and none of larger width, as shown in Table~\ref{table:6points-I}. Both are non-dps.
\end{theorem}

%% file: 3,1-coplanarity.tex
%%!TEX root =articulo6.tex

\section{Cases B and C: Polytopes containing a $(3,1)$-circuit (but no five coplanar points)}
\label{sec:(3,1)coplanarity}
Without loss of generality, we assume they contain the standard $(3,1)$-circuit: $p_1=o$, $p_2=e_1$, $p_3=e_2$ and $p_4=-e_1-e_2$. We treat separately the case of the other two points lying on the same side or on opposite sides of this circuit.

\medskip

\subsection{Case B: Polytopes with a $(3,1)$-circuit and the other two points on opposite sides} 
\label{sec:(3,1)coplanarity1}
In this case  $P^5=\conv (A \setminus\{p_5\})$ and $P^6=\conv (A \setminus\{p_6\})$ have to be polytopes of size $5$ and signature $(3,1)$. By Lemma~\ref{lemma:(3,1)thm}, $p_5$ and $p_6$ are both at lattice distance $1$ or $3$ from the plane of the coplanarity. This produces three cases, that we treat separately:
\medskip

\subsubsection{Both points at distance one}
\label{sec:B.i}
Without loss of generality we take $p_5=(0,0,1)$ and $p_6=(a,b,-1)$, $a,b \in \Z$ and use the \emph{parallel-planes} method. 

The configuration is  contained in the three planes $z=-1,0,1$, with
a single point in each of $z=\pm 1$, and $P\cap \{z=0\}$ is the convex hull of the $(3,1)$-circuit plus the mid-point $(a/2,b/2,0)$ of the edge $p_5p_6$.
Without loss of generality (because of the $S_3$-symmetries present in $P^6$) the intersection point can be assumed to be in the region $0 \le x \le y$. In order for $(1,1,0)$ and $(0,2,0)$ not to be in $P\cap \{z=0\}$, the region is bounded by $x<1$ and $y<3x+2$ (non-shaded area in Figure~\ref{fig:CaseGH}). This gives ten possibilities for the pair $(a,b)$, namely:
\begin{center}
\begin{tabular}{|c|c|c|cc|cc|c|c|c|c|}
\hline
\textbf{ Id.} & $B.1$ & $B.2$ & $B.3$ & $B.4$ & $B.5$ & $B.6$ & $B.7$ & $B.8$ & $B.9$ & $B.10$\\
\hline
$a$ & $1$ & $0$ & $1$ & $1$ & $1$ & $1$ & $0$ & $0$ & $1$ & $0$\\
\hline
$b$ & $4$ & $3$ & $2$ & $3$ & $5$ & $6$ & $2$ & $1$ & $1$ & $0$\\
\hline
\end{tabular}
\end{center}

All the options have automatically size $6$ since no more points arise at $P \cap \{z=0\}$. As shown in Table~\ref{table:6points-I}, they all have different volume vectors and so they give different equivalent classes.

Notice that the intersection point of the edge $p_5p_6$ with the plane $z=0$ determines the oriented matroid, depending on whether it is contained or not in one or more straight lines spanned by the edges in $z=0$. The ten configurations in the table are separated according to their oriented matroid.
\medskip

\subsubsection{One point at distance one and the other at distance three}
\label{sec:B.ii}

Without loss of generality $p_5=(0,0,1)$ and $p_6=(a,b,-3)$ with (by Lemma~\ref{lemma:(3,1)thm}) $a \equiv -b \equiv \pm 1 \pmod 3$.  

The configuration is contained between the hyperplanes $z=1$ and $z=-3$. The intersection point of the edge $p_5p_6$ with $z=0$ is $(a/4,b/4,0)$. As before, the intersection point can be assumed in the region bounded by $0 \le x \le y$, $x<1$ and $y<3x+2$. This gives us $44$ options for the pair $(a,b)$, displayed in Figure~\ref{fig:CaseGH}, and separated according to the oriented matroid in the following table:

\begin{center}
\begin{tabular}{|c|c|c|c|c|c|c|c|c|c|ccc|}
\hline
$a$ & $0$ & $0$ & $0$ & $0$ & $1$ & $1$ & $1$ & $2$ & $3$ & $1$ & $2$ & $3$\\  
\hline
$b$ & $0$ & $\{1,2,3\}$ & $4$ & $\{5,6,7\}$ & $1$ & $2$ & $3$ & $2$ & $3$ & $6$ & $8$ & $10$ \\  
\hline
\end{tabular}
\end{center}

\begin{center}
\begin{tabular}{|c|ccc|ccc|}
\hline
$a$ & $1$ & $2$ & $3$ & $1$ & $2$ & $3$\\  
\hline
$b$ & $\{4,5\}$ & $\{3,...,7\}$ & $\{4,...,9\}$ & $\{7,...,10\}$ & $\{9,...,13\}$ & $\{11,...,16\}$\\  
\hline
\end{tabular}
\end{center}

\begin{figure}[htb]
\begin{center}
\includegraphics[scale=1]{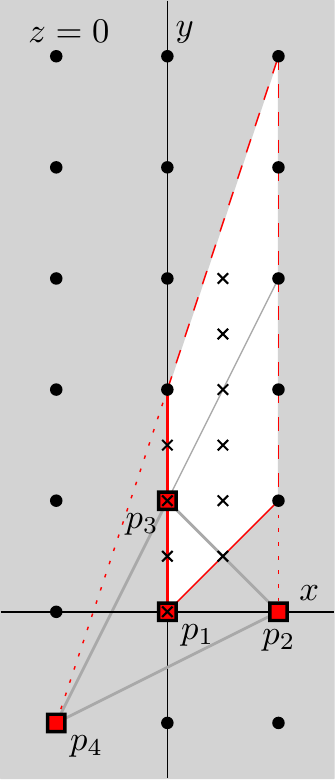} \qquad\qquad \includegraphics[scale=1]{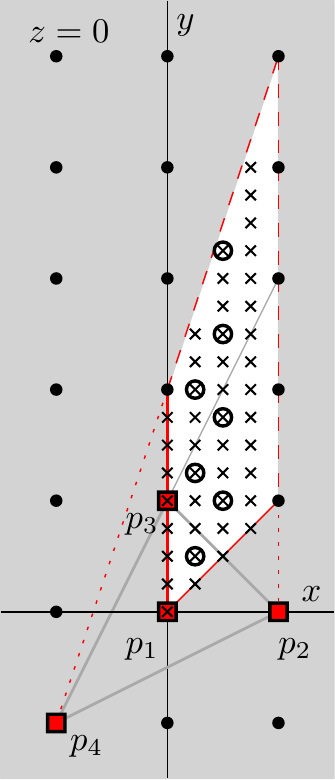}\\
\quad Sec.~\ref{sec:B.i} \hskip 3.5cm Sec.~\ref{sec:B.ii}
\caption{The cases of sections~\ref{sec:B.i} and~\ref{sec:B.ii}. Red squares represent the points $p_1$, $p_2$, $p_3$ and $p_4$ of $P$ in the displayed plane $z=0$. Black dots are the lattice points in the plane and black crosses represent the possible intersection points of the edge $p_5p_6$ and the plane $z=0$. In case~\ref{sec:B.ii}, the circled black crosses are the possibilities for which  $a \equiv -b \equiv \pm 1 \pmod 3$, as required by Lemma~\ref{lemma:(3,1)thm}.}
\label{fig:CaseGH}
\end{center}
\end{figure}

Among these $44$ possibilities, only the following five satisfy $a \equiv -b \equiv \pm 1 \pmod 3$, as required by Lemma~\ref{lemma:(3,1)thm}, and $\gcd(a,b,4)=1$, as required for the edge $p_5p_6$ to be primitive:

\begin{center}
\begin{tabular}{|c|c|cc|cc|}
\hline
$a$ & $1$ & $1$ & $2$ & $1$ & $2$\\
\hline
$b$ & $2$ & $5$ & $7$ & $8$ & $13$\\
\hline
\end{tabular}
\end{center}

Finally, we need to check which of this five possibilities do not produce additional lattice points in the planes
$z=-1$ or $z=-2$. For this we consider the following triangulations of $P$, depending on the oriented matroid, and check emptiness of each tetrahedron in the triangulation:

\begin{itemize}
\item If $(a,b)=(1,2)$, a triangulation is $P= P^5 \cup P^6$ and, by construction, $P$ has size $6$. 
\item If $(a,b)=(1,5)$ or $(2,7)$, a triangulation is $P= P^5 \cup P^6 \cup T_{2356}$. $T_{2356}$ is empty for $(1,5)$, but not for $(2,7)$.
\item If $(a,b)=(1,8)$ or $(2,13)$, a triangulation is $P= P^5 \cup P^6 \cup T_{2356} \cup T_{3456}$. Both $T_{2356}$ and $T_{3456}$ are empty for $(1,8)$, but $T_{2356}$ is not empty for $(2,13)$.
\end{itemize}

The three (non-equivalent) possibilities for $(a,b)$ that give size $6$ are $(1,8)$, $(1,2)$ and $(1,5)$, labeled B.11, B.12 and B.13 in Table~\ref{table:6points-I}.
\smallskip

\subsubsection{Both points at distance three}
\label{sec:B.iii}

Without loss of generality we take $p_5=(1,2,3)$ and $p_6=(a,b,-3)$ with $a \equiv -b \equiv \pm 1 \pmod 3$. 

\noindent In this case the configuration is contained between the hyperplanes $z=3$ and $z=-3$, and the intersection point of the edge $p_5p_6$ with $z=0$ is $(\frac{a'}{2},\frac{b'}{2},0)=(\frac{a+1}{2},\frac{b+2}{2},0)$. In this case, $P^6$ has less symmetries than before and we can only assume without loss of generality that this intersection point lies in $x,y \ge 0$. This time, in order for $(2,0,0)$, $(1,1,0)$ and $(0,2,0)$ not to be in $P\cap \{z=0\}$, the region is divided in two: either $x<1$ and $y<3x+2$, or $y<1$ and $x<3y+2$. This gives us $18$ options for the pair $(a',b')$, displayed in Figure~\ref{fig:CaseI} and separated according to the oriented matroid in the following table:

\begin{figure}[htb]
\begin{center}
\raisebox{-0.5\height}{\includegraphics[scale=0.9]{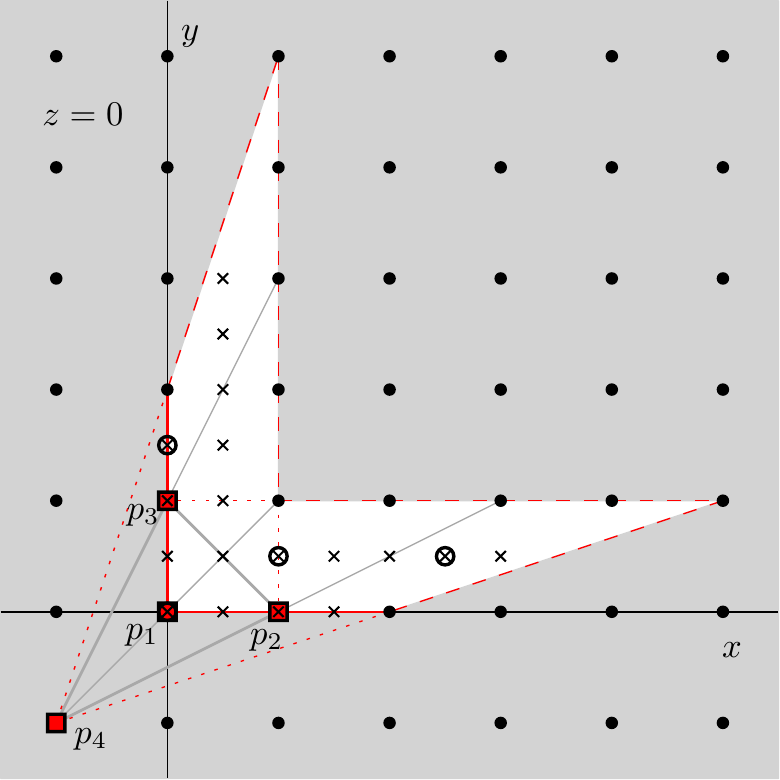}}
\qquad \qquad
\begin{tabular}{|c|c|}
  \hline
   $a'$& $b'$\\
\hline
  \ $0$ \ & $0$ \\
\hline
  \ $0$ \ & $1$ \\
  \ $1$ \ & $0$ \\
\hline
  \ $0$ \ & $2$ \\
  \ $2$ \ & $0$ \\
\hline
  \ $0$ \ & $3$ \\
  \ $3$ \ & $0$ \\
\hline
  \ $1$ \ & $1$ \\
\hline
  \ $1$ \ & $\{2,3\}$ \\
  \ $\{2,3\}$ \ & $1$ \\
\hline
  \ $1$ \ & $4$ \\
  \ $4$ \ & $1$ \\
\hline
  \ $1$ \ & $\{5,6\}$ \\
  \ $\{5,6\}$ \ & $1$ \\
\hline
\end{tabular}
\caption{The case of section~\ref{sec:B.iii}. Red squares represent the points $p_1$, $p_2$, $p_3$ and $p_4$ of $P$ in the displayed plane $z=0$. Black dots are the lattice points in the plane and black crosses represent the possible intersection points of the edge $p_5p_6$ and the plane $z=0$. The circled black crosses are the remaining possibilities after demanding $a \equiv -b \equiv \pm 1 \pmod 3$ (Lemma~\ref{lemma:(3,1)thm}).}
\label{fig:CaseI}
\end{center}
\end{figure}

After considering the restrictions $a \equiv -b \equiv \pm 1 \pmod 3$ we are left with only $4$ possibilities:

\begin{center}
\begin{tabular}{|c|c|c|c|c|}
\hline
$(a',b')$ & $(0,0)$ & $(0,3)$ & $(2,1)$ & $(5,1)$\\
\hline
$(a,b)$ & ${\bf (-1,-2)}$ & ${\bf (-1,1)}$ & $(1,-1)$ & $(4,-1)$\\
\hline
\end{tabular}
\end{center}

We can discard $(a,b)=(1,-1)$ and $(a,b)=(4,-1)$ because in this cases $\gcd(a-1,b-2,6)=3$, which implies the edge $p_5p_6$ has lattice points at heights $\pm 1$. This leaves us with two possibilities, marked in bold face in the previous table, and in both we verify that $P$ has no extra lattice points by triangulating it:

\begin{itemize}
\item If $(a',b')=(0,0)$, a triangulation is $P= P^5 \cup P^6$ so $P$ has automatically size $6$. 
\item If $(a',b')=(0,3)$, a triangulation is $P= P^5 \cup P^6 \cup T_{2356} \cup T_{3456}$ and both $T_{2356}$ and $T_{3456}$ are empty.
\end{itemize}

The two possibilities are labelled B.14 and B.15 in Table~\ref{table:6points-I}.

\medskip

\subsection{Case C: Polytopes with a $(3,1)$-circuit and the other two points on the same side} 
\label{section:3-1-sameside}
Observe that a configuration may contain more than one $(3,1)$-circuit. If one of them leaves the other two points in opposite sides we have already dealt with it, so here we assume that 
all the $(3,1)$ coplanarities leave the two other points at the same side. 

 Without loss of generality, we take the two extra points  in the halfspace $z >0$ and with $p_5$ having $z$-coordinate less than or equal to $p_6$. 
 We distinguish the cases where both $p_5$ and $p_6$ are vertices of $P$, or $p_5$ is in the convex hull of the other five points.

\subsubsection{Only $p_6$ is a vertex}
\label{sec:C.i}

Then, $p_5$ lies in the convex hull of $P^5=\conv (A\setminus \{p_5\})$. A priori it can lie either along an edge $p_ip_6$, in the interior of a triangle $p_ip_jp_6$, or in the interior of a tetrahedron $p_ip_jp_kp_6$, where $i,j,k\in\{1,2,3,4\}$.

\begin{lemma}
If $p_5$ lies in the interior of a triangle then we are in case B.
\end{lemma}

\begin{proof}
If the triangle includes $p_1$ then we are in case B. So, we assume without loss of generality that
$p_5 \in p_2p_3p_6$, an exterior facet of $P$. Then, since $p_1$ and $p_5$ are the centroids of the facets $p_2p_3p_4$ and $p_2p_3p_6$ respectively, the points $p_1$, $p_4$, $p_5$ and $p_6$ form a convex quadrilateral that is not a parallelogram (the lines $p_1p_4$ and $p_5p_6$ intersect in the mid-point of the edge $p_2p_3$). Such a lattice quadrilateral must contain extra lattice points.
\end{proof}

Thus, we only need to study the cases where $p_5$ is along an edge or in the interior of a tetrahedron.
In what follows we take without loss of generality $p_5=(0,0,1)$ or $p_5=(1,2,3)$ and,
because of the rotation symmetries of the $(3,1)$-circuit, we assume that $p_5$ is in the tetrahedron  $T_{1236}$.

If $p_5$ lies along a segment $p_ip_6$, then without loss of generality, $p_5 \in p_2p_6$ or $p_5 \in p_1p_6$ (that is, $p_6=2p_5-p_2$ or $p_6=2p_5-p_1$). Together with the two possibilities for $p_5$ this gives four possible configurations, that we analyze one by one:

\begin{itemize}
\item $p_5=(0,0,1)$ and $p_6=2p_5-p_2=(0,0,2)-(1,0,0)=(-1,0,2)$. In this case, a triangulation is $P=P^6 \cup T_{3456}$. $P^6$ is a polytope of size $5$ and $T_{3456}$ is an empty tetrahedron for this values of $p_5$ and $p_6$, so $P$ has size $6$. 
This gives configuration C.1 in Table~\ref{table:6points-I}.
%\[C.1\bordermatrix{ & 1 & 2 & 3 & 4 & 5 &6 \cr 
%    &0 & 1 & 0 &-1 & 0 & -1\cr 
%    &0 & 0 & 1 &-1 & 0 & 0\cr 
%    &0 & 0 & 0 & 0 & 1 & 2\cr }\]
%

\item $p_5=(1,2,3)$ and $p_6=2p_5-p_2=(2,4,6)-(1,0,0)=(1,4,6)$. The triangulation is the same as before: $P=P^6 \cup T_{3456}$, but in this case $T_{3456}$ is not empty.

\item $p_5=(0,0,1)$ and $p_6=2p_5-p_1=(0,0,2)$. In this case, a triangulation is $P=P^6 \cup T_{2356}\cup T_{2456}\cup T_{3456}$. $P^6$ is a polytope of size $5$ and all of $T_{2356}$, $T_{2456}$, $T_{3456}$ are empty tetrahedra for this values of $p_5$ and $p_6$, so $P$ has size $6$.
This gives configuration C.2 in Table~\ref{table:6points-I}.
%\[C.2 \bordermatrix{ & 1 & 2 & 3 & 4 & 5 &6 \cr 
%    &0 & 1 & 0 &-1 & 0 & 0\cr 
%    &0 & 0 & 1 &-1 & 0 & 0\cr 
%    &0 & 0 & 0 & 0 & 1 & 2\cr }\]
%

\item $p_5=(1,2,3)$ and $p_6=2p_5-p_1=(2,4,6)$. The triangulation is the same as before: $P=P^6 \cup T_{2356}\cup T_{2456}\cup T_{3456}$, and again all three tetrahedra are empty. 
This gives configuration C.3 in Table~\ref{table:6points-I}.
%\[C.3 \bordermatrix{ & 1 & 2 & 3 & 4 & 5 &6 \cr 
%    &0 & 1 & 0 &-1 & 1 & 2\cr 
%    &0 & 0 & 1 &-1 & 2 & 4\cr 
%    &0 & 0 & 0 & 0 & 3 & 6\cr }\]
%
\end{itemize}

If $p_5$ lies in the interior of the tetrahedron $T_{1236}$, observe that $p_5$ cannot be coplanar with $p_1p_4p_6$ because then these points would form a convex quadrilateral with the edges $p_1p_4$ and $p_5p_6$ not parallel, which implies extra lattice points. We assume without loss of generality that $p_5$ is on the same side of this plane as $p_2$. This fixes the oriented matroid of $P$ to be $5.4$*, with the following labeling (on the right we list its circuits):

\begin{center}
\includegraphics[width=0.4\textwidth]{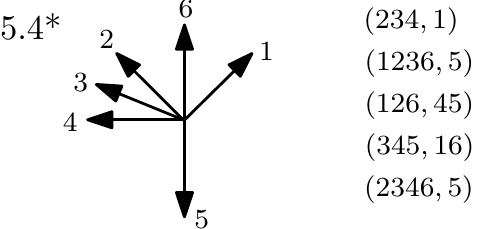}
\end{center}

Since $P^4$ has to be one of the eight configurations of size $5$ and signature $(4,1)$ displayed in Table~\ref{table:5points} and since $p_4 = 3p_1 -p_2 -p_3$, we use the \emph{$(4,1)$-extension} method computationally. The number of possibilities is a priori $8\times 4! = 196$, but most of them can easily be discarded because they do not produce the correct oriented matroid or because the subtetrahedron $p_1p_2p_3p_5$ does not have volume $3$ or $1$.
Among those who satisfy these conditions, the absence of extra lattice points is equivalent to emptiness of the tetrahedra $T_{1246}$ and $T_{1346}$ since these two tetrahedra, together with $P^4$, triangulate the configuration.

There are $2$ equivalence classes that pass this test, labeled C.4 and C.5 in Table~\ref{table:6points-I}.

\medskip

\subsubsection{Both $p_5$ and $p_6$ are vertices}
\label{sec:C.ii}

Then both $p_5$ and $p_6$ must be at lattice distance $1$ or $3$ from the coplanarity. If both are at distance $1$ then $P$ has width one. If both are at distance $3$ then the volume vector of $P$ must be divisible by three, which contradicts the fact that $P^2$, $P^3$ and $P^4$ are polytopes of size five and signature $(3,2)$ (such polytopes must have unimodular subtetrahedra, by the classification in Table~\ref{table:5points}).

Hence, we assume $p_5$ to be at distance $1$ and $p_6$ at distance $3$ from the coplanarity. Without loss of generality we take $p_6=(1,2,3)$ and $p_5=(a,b,1)$ for some $a,b \in \Z$.

The configuration is contained in the hyperplanes $z=0,1,3$, so we first use the \emph{parallel-planes} method to check which coordinates for $p_5$ do not produce more lattice points in the plane $z=1$. The intersection points of edges $p_1p_6$, $p_2p_6$, $p_3p_6$ and $p_4p_6$ with $z=1$ are, respectively, $(1/3,2/3,1)$, $(1,2/3,1)$, $(1/3,4/3,1)$ and $(-1/3,0,1)$. We want to check which values of $a,b$ leave $p_5$ outside the convex hull of those four intersection points, and do not produce more lattice points in $P$. Without loss of generality, because of the rotation symmetries of $P^5$, we can assume that $p_5$ lies in $x \ge 1/3$, $y \ge 2/3$ (non-shaded area in Figure~\ref{fig:CaseM}).

\begin{figure}[htb]
\begin{center}
%$\bordermatrix{ & 1 & 2 & 3 & 4 & 5 &6 \cr 
%    &0 & 0 & -1 &1 & 1 & 1\cr 
%    &0 & 1 & -1 &0 & 1 & 2\cr 
%    &0 & 0 & 0 & 0 & 1 & 3\cr}$
\includegraphics[scale=1.10]{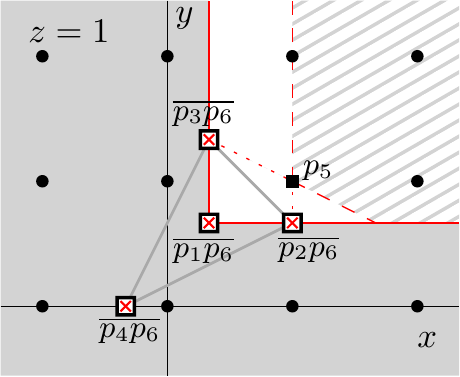}
\caption{The case of Section~\ref{sec:C.ii}. The red crossed squares are the intersection of edges $p_1p_6$, $p_2p_6$, $p_3p_6$ and $p_4p_6$ with the displayed plane $z=1$. Black dots are the lattice points in the plane and black squares represent the possible lattice points for $p_5$.}
\label{fig:CaseM}
\end{center}
\end{figure}

Suppose $a>1$ or $b>1$. Then the convex hull of $P$ at $z=1$ will enclose $(1,1,1)$ as a seventh lattice point. Hence the only possibility is that $a=1=b$.

It remains to check whether any lattice points at $z=2$ lie inside $P$. A triangulation of $P$ is $P= P^5 \cup T_{2356}$. We already know $P^5$ to be a polytope of size $5$, and $T_{2356}$ is indeed empty. Hence we get a configuration of size $6$, which is labelled C.6 in Table~\ref{table:6points-I}.

\bigskip

We now analyze the width of the six configurations obtained in case C. The only functional that can possibly produce width one is the functional $z$: every other functional is non-constant in the plane $z=0$ containing the $(3,1)$ coplanarity and the subconfiguration in this plane has an interior point.
Hence, all the configurations have width at least two. In all of them we give functionals which achieve this width except in configuration C.3, where we give one producing width 3. Let us show that this configuration does not have width two:

\begin{lemma}
The  configuration $A$ consisting of the columns of the following matrix has width at least three:
\[%C.3
\left( \begin{array}{cccccc}
  0 & 1 & 0 &-1 & 1 & 2\\ 
  0 & 0 & 1 &-1 & 2 & 4\\ 
  0 & 0 & 0 & 0 & 3 & 6\end{array} \right)\]
\end{lemma}

\begin{proof}
We label the six points $p_1$, \dots, $p_6$ in the given order. Observe that $p_5$ is the mid-point of segment $p_1p_6$
and $p_1$ is the centroid of triangle $p_2p_3p_4$. Hence, a functional giving width two to $A$ must be constant either on the triangle or the segment. 

The only (modulo sign)  linear, primitive, functional constant on the triangle is $f(x,y,z)=z$, which has width six on $A$. Hence, for the rest of the proof we assume $f$ to be constant (and zero) on the segment $p_1p_6$. Let $f_0(x,y):=f(x,y,0)$ be $f$ restricted to the plane $\{z=0\}$. $f_0$ must be one of the three functionals giving width two to the $(3,1)$ configuration on that plane, namely $f_0(x,y)=x$, $f_0(x,y)=y$, or $f_0(x,y)=x-y$. The extensions of these functionals that are constant on the segment are $f(x,y)=x-z/3$, $f(x,y)=y-2z/3$, or $f(x,y)=x-y+z/3$ which are not integer. Hence, there exists no integer functional giving width two to $A$.
\end{proof}

Summing up cases B and C:

\begin{theorem}
\label{thm:6points-2}
Among the lattice $3$-polytopes of size six with no $5$ coplanar points but with some $(3,1)$ coplanarity, there are exactly $20$ equivalence classes of width two, $1$ of width three, and none of width larger than three, as shown in Tables~\ref{table:6points-I}. $13$ of those are dps, all of them of width two.
\end{theorem}

%% file: 2,2-coplanarity.tex
%%!TEX root =articulo6.tex

\section{Cases D and E: Polytopes containing a $(2,2)$-circuit (but no $(3,1)$-circuit, and no five coplanar points)}
\label{sec:(2,2)coplanarity}

These are clearly non-dps configurations. Without loss of generality, we assume they contain the standard  $(2,2)$-circuit: $p_1=o$, $p_2=e_1$, $p_3=e_2$ and $p_4=e_1+e_2$.
Again, we treat separately the possibilities for the side of the coplanarity the other two points lie in.

\subsection{Case D: Polytopes with a $(2,2)$-circuit and the other two points on opposite sides} 
\label{sec:(2,2)coplanarity1}

Then $P^5$ and $P^6$ must be polytopes of size $5$ and signature $(2,2)$. By Table~\ref{table:5points}, $p_5$ and $p_6$ must be at lattice distance one from the $(2,2)$-circuit. Without loss of generality we take $p_5=(0,0,1)$, and $p_6=(a,b,-1)$ for $a,b \in \Z$. 

The configuration is contained in the three planes $z=-1,0,1$, so we use the parallel-planes method to check if the size of $P$ is $6$. There is one single point in the planes $z=\pm 1$, and $P\cap \{z=0\}$ is the convex hull of the $(2,2)$-circuit and the intersection point $(a/2,b/2,0)$ of the edge $p_5p_6$ with the plane $z=0$.

Without loss of generality (because of the symmetries present in $P^6$) the intersection point can be assumed to be in the region $1/2 \le x \le y$. In order for $(1,2,0)$ and $(2,2,0)$ not to be in $P\cap \{z=0\}$, either $1/2 \le x<1$, or $x \le y<x+1$ (non-shaded area in Figure~\ref{fig:CaseN}). This gives an infinite number possibilities for the pair $(a,b)$, but those with $a=1$ or $a=2$ can be discarded since they have width one (with respect to the functional $x+z$, for example). Only the three possibilities in the table of  Figure~\ref{fig:CaseN} remain. They are separated according to oriented matroid.

\begin{figure}[htb]
\begin{center}
\raisebox{-0.5\height}{\includegraphics{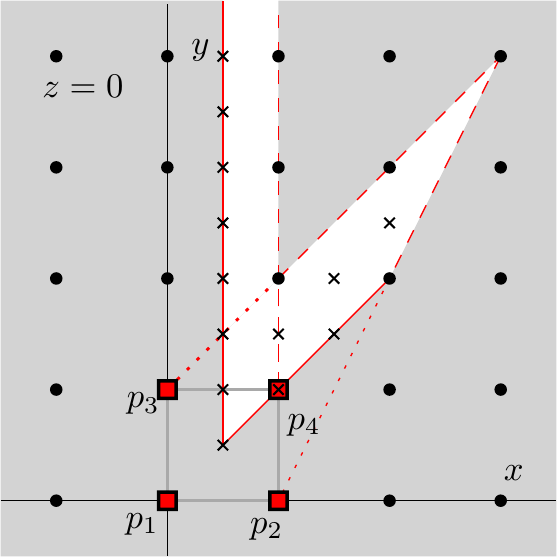}}
\qquad \qquad
\begin{tabular}{|c|c|}
  \hline
   $a$& $b$\\
\hline
  \ $3$ \ & $3$ \\
\hline
  \ $3$ \ & $4$ \\
  \ $4$ \ & $5$ \\
\hline
\end{tabular}
\caption{The analysis of case D. Red squares represent the points $p_1$, $p_2$, $p_3$ and $p_4$ of $P$ in the displayed plane $z=0$. Black dots are the lattice points in the plane and black crosses represent the possible intersection points of the edge $p_5p_6$ and the plane $z=0$.}
\label{fig:CaseN}
\end{center}
\end{figure}

These three options have automatically size $6$ since no more points arise at $P\cap \{z=0\}$. But the configuration with $(a,b)=(3,3)$ contains a $(3,1)$ circuit ($p_4$ is the barycenter of $p_1$, $p_5$ and $p_6$) so it has already been considered. 

The two remaining configurations, with $(a,b)=(3,4)$ and $(a,b)=(4,5)$, have the same oriented matroid but different volume vectors. These are configurations D.1 and D.2 in Table~\ref{table:6points-I}, and they clearly have width two.
\medskip

\subsection{Case E: Polytopes with a $(2,2)$-circuit and the other two points on the same side} 
\label{sec:(2,2)+(4,1)}

Contrary to what happened in the case of a $(3,1)$-circuit it is now impossible for $p_5$ and $p_6$ to be both vertices, because then they would both be at distance one from the $(2,2)$ coplanarity and the configuration would have width one. So, we assume without loss of generality that $p_6$ is a vertex and $p_5$ is not, and take $p_5=(0,0,1)$.

Because of the symmetries of the $(2,2)$-circuit, we can assume that $p_6$ is so that $p_5 \in T_{1236}$. Now, $p_5$ can neither be along an edge of $T_{1236}$ (because then $P$ has width one) nor in the interior of a facet (because then there would be a $(3,1)$-circuit, and this case has been already considered). It must thus be in the interior of $T_{1236}$. By symmetry, $p_5$ cannot be in the plane $p_1p_4p_6$ either, and we can assume it to be on the same side of this plane as $p_2$.

The oriented matroid of this configuration is then  $5.5$, with its elements labeled as follows (on the right is the list of circuits):

\begin{center}
\includegraphics[width=0.4\textwidth]{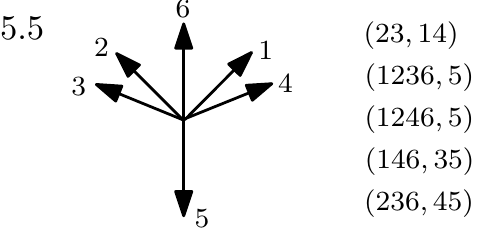}
\end{center}

This oriented matroid has one symmetry, the permutation $(12)(34)$. When comparing two of these configurations, we need to compare them as well after performing this permutation on one of them.

Then $P^4$ is one of the eight possible polytopes of size $5$ and signature $(4,1)$, and $p_4=p_2+p_3-p_1$, so we use the same technique as in the case $(3,1)$. We look at all the a-priori $8\times 4!$ ways of mapping $P^4$ to a configuration of size $5$ and signature $(4,1)$, and discard those for which the addition of $p_4 := p_2+p_3-p_1$ does not give a configuration of size six (equivalently, those for which the tetrahedron $p_2p_3p_4p_6$ turns out to be non-empty).
This results in $2$ equivalence classes, labeled E.1 and E.2 in Table~\ref{table:6points-I}. Since they have an interior point their width is at least two, and the table shows functionals that achieve this width.

\bigskip

In summary:

\begin{theorem}
\label{thm:6points-3}
Among the lattice $3$-polytopes of size six with no $5$ coplanar points, no $(3,1)$ coplanarity, but with some $(2,2)$ coplanarity, there are exactly $4$ equivalence classes of width two, and none of larger width, as shown in Table~\ref{table:6points-I}.
All of them are non-dps.
\end{theorem}

%% file: 2,1-coplanarity.tex
%%!TEX root =articulo6.tex

\section{Case F: Polytopes containing a $(2,1)$-circuit (but no other coplanarity)}
\label{sec:(2,1)coplanarity}

Suppose $A=\{r_1,r_2,r_3,u_1,u_2,u_3\}$ is a set of six lattice points, $r_2$ being the mid-point of $r_1$ and $r_3$, and such that $P=\conv A$ has size six and width $>1$. At least one of $r_1$ and $r_3$ must be a vertex of $P$ (since $r_2$ is not and $P$ has at least four vertices). Assume $r_3$ is a vertex, so that $P^3:=\conv(A\setminus \{r_3\})$ has to be a polytope of size five. Then:

\begin{lemma}
\label{lemma:2-1-weak}
$P^3$ has signature $(4,1)$ or $(3,2)$.
\end{lemma}

\begin{proof}
If $P^3$ has signature $(3,1)$ or $(2,2)$ we are in one of the cases B, C, D or E already considered. If $P^3$ has signature $(2,1)$ then either $P$ 
contains five coplanar points (case A, already considered) or it consists of $3+3$ points along two lines $l_1$ and $l_2$. The latter implies, by the classification on Table~\ref{table:5points}, that $P^3$ has width one precisely with respect to $l_1$ and $l_2$. But, in this case, $P$ has width one as well.
\end{proof}

\begin{lemma}
\label{lemma:2-1-strong}
If $P^3$ has signature $(3,2)$, then $r_1$ is a vertex and $P^1:=\conv(A\setminus \{r_1\})$ has signature $(4,1)$.
\end{lemma}

\begin{proof}
Suppose $P^3$ has signature $(3,2)$. If the segment $r_1r_2$ whose extension gives $r_3$ is the unique non-edge in $P^3$ then the statement clearly holds. So, let us assume that $r_1r_2$ is an edge in $P^3$. Hence, $r_1r_2r_3$ is an edge in $P$ and $r_1$ is indeed a vertex.

Now we have that both $P^3$ and $P^1$ are polytopes of size $5$ and signatures either $(4,1)$ or $(3,2)$ and need to show that at least one of them is $(4,1)$.
If some $u_i$ is not a vertex (say $u_2$) the proof is easy: $P$ is a tetrahedron with vertices $r_1r_3u_1u_3$, and it decomposes into the tetrahedra $P^1$ and $P^3$. $u_2$ cannot be in the boundary of any of $P^1$ or $P^3$, because that would imply a coplanarity, so it lies in the interior of one of them, which is then of signature $(4,1)$, as we wanted to show.

It remains to prove the lemma when all $u_1$, $u_2$ and $u_3$ are vertices. Assume, to get a contradiction, that both $P^3$ and $P^1$ are of signature $(3,2)$. 

Let us consider the projection of $P$ to a plane along the direction of the edge $r_1r_2r_3$, and let us denote by $\overline{r}$ the projection of this edge, and use $\overline{u_i}$ for the projection of each point $u_i$. Considering what we said above, and taking into account that we do not want $5$ coplanar points, it is easy to see that $\overline{u_1},\overline{u_2},\overline{u_3}$  lie in different rays with vertex $\overline{r}$, and that they are all contained in an open halfspace defined by a hyperplane passing through $\overline{r}$. Assume that $\overline{u_2}$ is in the middle ray of the three. There are three possibilities for the position of $\overline{u_2}$ relative to the line $\overline{u_1u_3}$, as shown in the figure below:\\

\begin{center}
\includegraphics[scale=1.25]{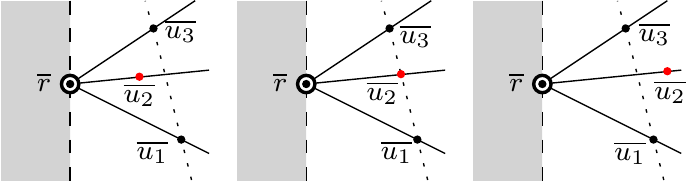}
\end{center}

Now, $P\setminus \{u_i\}:=\conv(A\setminus \{u_i\})$ has signature $(2,1)$ for each $i=1,2,3$ and its size is $5$ if and only if $r_1r_2r_3$ and $u_iu_j$ are contained in parallel consecutive lattice planes for $i,j \in\{1,2,3\}$, $i \neq j$. In the projection, this implies that each $\overline{u_i}\overline{u_j}$ must span a straight line at lattice distance $1$ from the point $\overline{r}$. Let us see each case separately:

\begin{itemize}[leftmargin=.5cm]
\item In the first case, $\overline{u_1u_3}$ cannot be at lattice distance $1$ from $\overline{r}$, since $\overline{u_2}$ is an interior point of the triangle $\overline{ru_1u_3}$.

\item In the second case, $\overline{u_1}$, $\overline{u_2}$ and $\overline{u_3}$ would all lie in the same straight line at lattice distance $1$ from $\overline{r}$, which implies that the segment $r_1r_2r_3$ is at lattice distance $1$ from the vertical plane projecting to $\overline{u_1}$, $\overline{u_2}$ and $\overline{u_3}$. That is, $P$ has width one.

\item In the third case, in particular we need to have lattice distance $1$ from $\overline{r}$ to the line spanned by $\overline{u_1u_3}$. Besides, if the edge $\overline{u_1u_3}$ were not primitive in the projection, then either $\overline{u_1u_2}$ or $\overline{u_2u_3}$ would not be at lattice distance $1$ from $\overline{r}$. In particular, $\overline{ru_1u_3}$ is a unimodular triangle, and so is $r_iu_1u_3$ for each $i=1,2,3$.

So without loss of generality we can now assume that $r_1=(1,0,0)$, $r_2=(0,0,0)$, $r_3=(-1,0,0)$ (proyection in the direction of functional $x$), $u_1=(0,1,0)$ and $u_3=(0,0,1)$ (see Figure~\ref{fig:CaseW}). For the sixth point we take coordinates $u_2=(a,b,c)$. Since its projection is $(b,c)$ we have $b,c>0$. By the symmetry of the configuration we assume without loss of generality that $b\ge c$ and $a \ge 0$.  

\begin{figure}[h]
\begin{center}
\includegraphics[scale=1.25]{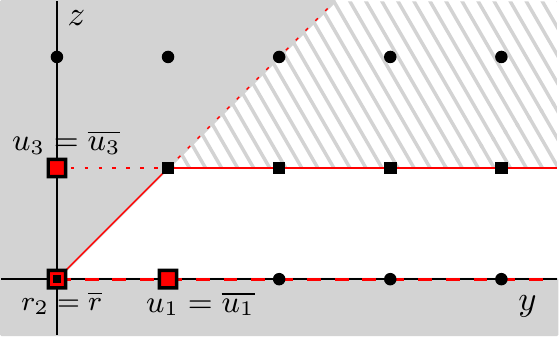}
\caption{The analysis of the third case. The projection onto the plane $x=0$ in the direction of edge $r_1r_2r_3$ is displayed. The double red square is the projection $\overline{r}=r_2$, and the other two red squares are $u_1=\overline{u_1}$ and $u_3=\overline{u_1}$. Black dots correspond to the projection of lattice points and black squares represent the possible projections of the point $u_2$.}
\label{fig:CaseW}
\end{center}
\end{figure}

Suppose now that $c >1$. Then the point $(1,1)$ is closer to the line spanned by $\overline{u_2u_3}$ than the point $\overline{r}$, contradicting the fact that $\overline{r}$ must be at distance $1$ from this line.  Hence we can only have $c=1$, and the configuration has width $1$ with respect to the functional $z$. 
\end{itemize}

\end{proof}

Thus, for the rest of this section we assume that $P^3$ has size $5$ and signature $(4,1)$.
There are three ways to extend a segment $r_1r_2$ of a polytope of size five and signature $(4,1)$, each corresponding to a different  oriented matroid. %Let $r_3=2r_2-r_1$:

\begin{itemize}
\item $r_1$ is the interior point of $P^3$, and $r_2$ is a vertex: oriented matroid  4.21.
\item $r_2$ is the interior point of $P^3$, and $r_1$ is a vertex: oriented matroid  4.22.
\item both $r_1$ and $r_2$ are vertices of $P^3$: oriented matroid 4.11.
\end{itemize}

We do this computationally: we take each one of the $8$ configurations of size $5$ and signature $(4,1)$ displayed in Table~\ref{table:5points}, choose a pair of vertices $\{r_1,r_2\}$, and check whether the configurations have size $6$ after adding the sixth point $r_3=2r_2-r_1$. 

There are $6$, $6$ and $5$ equivalence classes, respectively.

\subsubsection*{Oriented matroid  4.21}

\begin{center}
\includegraphics[width=0.4\textwidth]{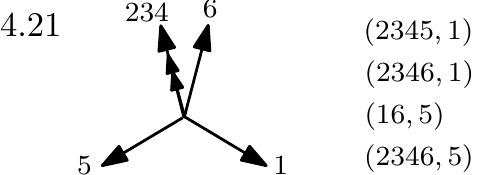}
\end{center}

We  organize the points $1:=p_1$ to $6:=p_6$ so that:
\begin{itemize}
\item $1$ is the interior point of the base polytope of signature $(4,1)$, namely the one with vertex set $2345$.
\item $(5,16)$ forms a $(2,1)$-circuit.
\item $2$, $3$ and $4$ are chosen in increasing order of the absolute value of the volume of $(P^6)^i$.
\end{itemize}
In terms of equivalence, the previous order is unique. 

A triangulation of $P$ is $P=P^6 \cup T_{2356} \cup T_{2456} \cup T_{3456}$. Since we chose $P^6$ to have size $5$, $P$ will have size $6$ if and only if $T_{2356}$,$T_{2456}$ and $T_{3456}$ are empty tetrahedra.
This gives configurations F.1 to F.6 in Table~\ref{table:6points-I}. 

\subsubsection*{Oriented matroid  4.22}

\begin{center}
\includegraphics[width=0.4\textwidth]{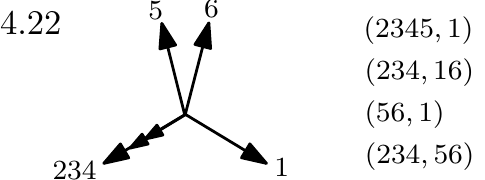}
\end{center}

We will organize the points $1:=p_1$ to $6:=p_6$ so that:
\begin{itemize}
\item $1$ is the interior point of the base polytope of signature $(4,1)$, namely the one with vertex set $2345$.
\item $(1,56)$ forms a $(2,1)$-circuit.
\item $2$, $3$ and $4$ are chosen in increasing order of the absolute value of the volume of $(P^6)^i$.
\end{itemize}
In terms of equivalence, the previous order is unique. 

A triangulation of $P$ is $P=P^6 \cup T_{2346}$. Since we chose $P^6$ to have size $5$, $P$ will have size $6$ if and only if $T_{2346}$ is an empty tetrahedron. This gives configurations F.7 to F.12 in Table~\ref{table:6points-I}.

\subsubsection*{Oriented matroid  4.11}

\begin{center}
\includegraphics[width=0.4\textwidth]{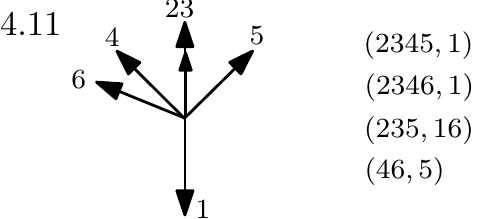}
\end{center}

We will organize the points $1:=p_1$ to $6:=p_6$ so that:
\begin{itemize}
\item $1$ is the interior point of the base polytope of signature $(4,1)$, namely the one with vertex set $2345$.
\item $(5,46)$ forms a $(2,1)$-circuit.
\item $2$ and $3$ are chosen in increasing order of the absolute value of the volume of $(P^6)^i$.
\end{itemize}

In terms of equivalence, the previous order is unique.

A triangulation of $P$ is $P=P^6 \cup T_{2356}$. Since we chose $P^6$ to have size $5$, $P$ will have size $6$ if and only if $T_{2356}$ is an empty tetrahedron. This gives configurations F.13 to F.17 in Table~\ref{table:6points-I}. 

%
%Notice that in the first case, $P\setminus \{r_1\}$ is a polytope of signature $(4,1)$ but with size $6$ ($r_1$ is an extra lattice point in its convex hull) and in the last two it is a polytope of size $5$ and signature $(3,2)$.

\bigskip
Let us now analyze the width of these configurations. In all the cases F.1 to F.17 we have (at least) one interior point so the width is at least $2$. We were able to find functionals giving width $2$ to all of them.

In summary:

\begin{theorem}
\label{thm:6points-4}
Among the lattice $3$-polytopes of size six with some $(2,1)$ coplanarity and no other coplanarity, there are exactly $17$ equivalence classes of width two, and none of larger width, as shown in Table~\ref{table:6points-I}. All of them are non-dps.
\end{theorem}

%% file: gp-intpoint.tex
%%!TEX root =articulo6.tex

\section{Cases H and G: Polytopes with no coplanarities}
\label{sec:gp-intpoint}

Since there are no coplanarities, these configurations are all dps. 
Apart from that, they must have unimodular triangles as facets and at least one interior point (otherwise they have width one by Howe's Theorem). 
There are only two uniform oriented matroids with interior points, namely, the oriented matroids 6.1 and 6.2 of Figure~\ref{fig:OM} (see also Figures~\ref{fig:62} and~\ref{fig:61} below).

Both of them happen to have two vertices $p_i$ and $p_j$ (the elements labeled 5 and 6 in the figures) such that $P^i$ and $P^j$ have signature $(4,1)$. Hence, the full classification of these polytopes can be done by an exhaustive exploration of all the possible ways to glue together two of the eight  polytopes of signature $(4,1)$ and size $5$ from Table~\ref{table:5points}. We do this computationally. We first choose a pair of these such polytopes and then we choose from each one of the four subtetrahedra. If both tetrahedra are of the same (or equivalent) type, there is one (or more) unimodular transformation that allows us to glue the two polytopes by four points. The result is a configuration of six lattice points and it only remains to check whether the convex hull has additional lattice points or not.

Now, an empty subtetrahedron in a polytope of signature $(4,1)$ consists of three of the vertices and the interior point. There are two ways of gluing these two tetrahedra (and respectively the polytopes of signature $(4,1)$), either making the interior point in both coincide or not. This leads to the two cases in this section.
\medskip

\subsection{Case G: Polytopes with no coplanarities and one interior point}
\label{sec:2*(4,1)-1}

The oriented matroid is $6.2$* (see Figure~\ref{fig:62}) and the two subpolytopes of signature $(4,1)$ share the same interior point.

\begin{figure}[htb]
\begin{center}
\includegraphics[width=0.4\textwidth]{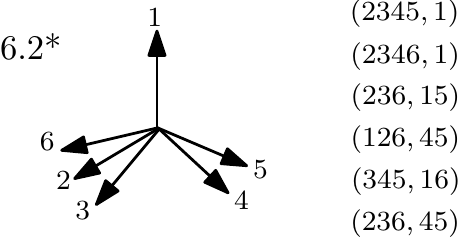}
\caption{The oriented matroid 6.2}
\label{fig:62}
\end{center}
\end{figure}

We  organize the points $1:=p_1$ to $6:=p_6$ so that:
\begin{itemize}
\item $5$ and $6$ are the points so that $P^5$ and $P^6$ are of signature $(4,1)$.
\item $1$ is the interior point of those polytopes.
\item $4$ is such that the configuration contains the circuit $(236,45)$.
\item $2$ is such that the configuration contains the circuit $(126,45)$. This leaves $3$ to be such that the configuration contains the circuit $(345,16)$.
\end{itemize}
In terms of equivalence, the previous order is unique.

A triangulation of $P$ is $P=P^5 \cup T_{2356}$. Since we chose $P^5$ to have size $5$, $P$ will have size $6$ if and only if $T_{2356}$ is an empty tetrahedron.
The computation shows that there are $20$ equivalence classes of size $6$, displayed in Table~\ref{table:6points-II}.

\medskip

\subsection{Case H: Polytopes with no coplanarities and two interior points}
\label{sec:2*(4,1)-2}

The oriented matroid is $6.1$* (see Figure~\ref{fig:62}) and the two subpolytopes of signature $(4,1)$ have different interior points: the interior point of one is a vertex of the other and viceversa.

This oriented matroid has one symmetry, the permutation $(14)(23)(56)$. When comparing two of these configurations, we need to compare them as well after performing this permutation on one of them.

\begin{figure}[htb]
\begin{center}
\includegraphics[width=0.4\textwidth]{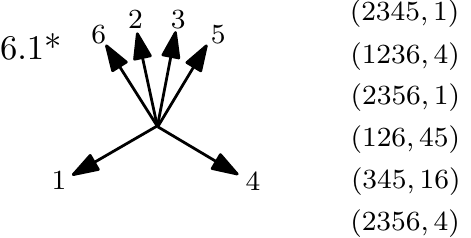}
\caption{The oriented matroid 6.1}
\label{fig:61}
\end{center}
\end{figure}

We will organize the points $1:=p_1$ to $6:=p_6$ so that:
\begin{itemize}
\item $5$ and $6$ are the points so that $P^5$ and $P^6$ are of signature $(4,1)$.
\item $1$ and $4$ are the interior points in $P^6$ and $P^5$, respectively.
\item $2$ is such that the configuration contains the circuit $(126,45)$. This leaves $3$ to be such that the configuration contains the circuit $(345,16)$.
\end{itemize}

With that ordering of the points, a triangulation of $P$ is:
\[
P=P^6 \cup T_{1256} \cup T_{1356} \cup T_{1456} \cup T_{2456} \cup T_{3456}
\]
Since we chose $P^6$ to have size $5$, $P$ will have size $6$ if and only if those five tetrahedra are empty.
The computation shows that there are $12$ equivalence classes of size $6$, displayed in Table~\ref{table:6points-II}.
\bigskip

Let us now analyze the width of the configurations G and H. Since they all contain an interior point, they all have width at least $2$. In all of them the tables show a functional of width two except in 
the one labeled H.12, for which we give one of width three. 
Let us show that this configuration does not have width two:

\begin{lemma}
The  configuration $A$ consisting of the columns of the following matrix has width at least three:
\[%H.12
\left( \begin{array}{cccccc}
 0  & 1  & 0  & 2  &-1 &  4\\
 0  & 0  & 0  & 5  &-2 & 11\\
 0  & 0  & 1  & 1  &-1 &  2\end{array} \right)\]
\end{lemma}

\begin{proof}
We label the six points $p_1$, \dots, $p_6$ in the given order. Remember that points $p_1$ and $p_4$ are in the interior of $\conv(A)$. Thus, a functional $f$ giving width two to $A$ must be constant on the segment $p_1p_4$. Consider the projection $\pi(x,y,z)=(x-2z,y-5z)$ sending that segment to the origin, and sending $\Z^3$ surjectively to $\Z^2$. If $f$ gives width two to $A$, it must factor via a linear functional $f_0:\Z^2\to \Z$ giving width two to $\pi(A)$, which is the following configuration:
\[\left( \begin{array}{cccccc}
 0  & 1  & -2  & 0  &1 &  0\\
 0  & 0  & -5  & 0  &3 & 1\end{array} \right)\]
So, we only need to show that this two-dimensional configuration has width at least three, which is easy: It contains the non-parallel segments $\conv\{(1,0),(1,3)\}$ and $\conv\{(-2,-5),(1,1)\}$, both of length three.
\end{proof}

In summary:

\begin{theorem}
\label{thm:6points-5}
Among the lattice $3$-polytopes of size six with no coplanarities and at least one interior point, there are exactly $31$ equivalence classes of width two, $1$ of width three, and none of larger width, as shown in Tables~\ref{table:6points-II}.
All of them are dps.
\end{theorem}

%% file: tables.tex
%%!TEX root =articulo6.tex

\section{Tables}
\label{sec:tables}

We here list all the $3$-polytopes of size six and width greater than one. They are contained in Tables~\ref{table:6points-I} and~\ref{table:6points-II}, where we give, for each of them, its oriented matroid, a reference ID related to where in this paper it is obtained, its volume vector, its width, and a functional where its width is achieved. In what follows we also give a $3\times 6$ integer matrix for each of them, whose columns are the six lattice points in a representative of the class. The volume vector in the tables is always given with respect to the order of columns in the corresponding matrix.

\begin{remark}
Among the $76$ polytopes of size $6$ and width $>1$, all except A.1, A.2, B.14, B.15 and C.3 contain a unimodular tetrahedron. In particular their volume vectors are primitive and, in the light of Theorem~\ref{thm:VolumeVectors}, they characterize the configurations uniquely. 

For the other five the same is not true. It is easy to show that each of the configurations A.1, A.2, B.14, B.15 and C.3 has 
the same volume vector as one containing extra points in its convex hull.
For example, B.14, B.15 and C.3 are affinely equivalent to B.2, B.10 and C.2, respectively, via affine maps of determinant $1/3$. Dilating the $x$ direction in the latter we get configurations with the volume vectors of the former, but with extra points in their convex hulls.
The situation for A.1 and A.2 is similar, except now the volume vectors are divisible by two instead of three. The configurations obtained dividing their volume vector by two have width one.
\end{remark}

\setlength{\tabcolsep}{3pt}
\setlength{\arraycolsep}{3pt}

\begin{table}[htbp]
\centering
\small
\begin{tabular}{|c|l|rrrrrrrrrrrrrrr|ccc|} 
\hline
\textbf{OM} &\textbf{Id.} &\multicolumn{15}{c|}{\textbf{Volume vector}} &\multicolumn{3}{c|}{\torre{ \textbf{Width,}}{ \textbf{functional}}} \\
\hline
\multicolumn{20}{|c|}{ \textbf{Polytopes containing $5$ coplanar points}}\\
\hline
3.2&A.1& $0$ & $0$ & $2$ & $0$ & $0$ & $4$ & $0$ & $2$ & $0$ & $-4$ & $0$ & $4$& $-2$ & $-8$ & $-2$ &\ &$2$&$z$\\
\hline
3.3&A.2& $0$ & $0$ & $2$ & $0$ & $4$ & $4$ & $0$ & $2$ & $0$ & $-4$ & $0$ & $0$ & $-2$ & $-4$ & $-2$ &\ &$2$ & $z$\\
\hline
\hline
\multicolumn{20}{|c|}{ \textbf{Polytopes containing a $(3,1)$ coplanarity, but no $5$ coplanar points}}\\
\hline
\multicolumn{20}{|c|}{One $(3,1)$ coplanarity leaves the other two points at opposite sides}\\
\hline
3.8&B.7 & $0$ & $1$ & $-1$ & $-1$ & $1$ & $-2$ & $1$ & $-1$ & $0$ & $2$ & $3$ & $-3$ & $0$ & $6$ & $0$  &\ &$2$& $z$\\
\hline
3.9&B.9& $0$ & $1$ & $-1$ & $-1$ & $1$ & $-1$ & $1$ & $-1$ & $1$ & $0$ & $3$ & $-3$ & $0$ & $3$ & $-3$  &\ &$2$& $z$\\
\hline
3.13&B.10& $0$ & $1$ & $-1$ & $-1$ & $1$ & $0$ & $1$ & $-1$ & $0$ & $0$ & $3$ & $-3$ & $-2$ & $2$ & $-2$  &\ &$2$& $z$\\
&B.15& $0$ & $3$ & $-3$ & $-3$ & $3$ & $0$ & $3$ & $-3$ & $0$ & $0$ & $9$ & $-9$ & $-6$ & $6$ & $-6$  &\ &$2$&$x$\\
\hline
4.13*&B.1& $0$ & $1$ & $-1$ & $-1$ & $1$ & $-4$ & $1$ & $-1$ & $1$ & $3$ & $3$ & $-3$ & $3$ & $9$ & $0$  &\ &$2$&$z$ \\
\hline
4.17*&B.2& $0$ & $ 1$ & $-1$ & $-1$ & $ 1$ & $-3$ & $ 1$ & $-1$ & $ 0$ & $ 3$ & $ 3$ & $-3$ & $ 1$ & $ 8$ & $ 1$  &\ &$2$& $z$ \\
&B.14& $0$ & $ 3$ & $-3$ & $-3$ & $ 3$ & $-9$ & $ 3$ & $-3$ & $ 0$ & $ 9$ & $ 9$ & $-9$ & $ 3$ & $24$ & $ 3$ &\ &$2$& $x$\\
\hline
4.18&B.8& $0$ & $1$ & $-1$ & $-1$ & $1$ & $-1$ & $1$ & $-1$ & $0$ & $1$ & $3$ & $-3$ & $-1$ & $4$ & $-1$ &\ &$2$& $z$\\
\hline
&B.5& $0$ & $1$ & $-1$ & $-1$ & $1$ & $-5$ & $1$ & $-1$ & $1$ & $4$ & $3$ & $-3$ & $4$ & $11$ & $1$ &\ &$2$& $z$ \\
5.10*&B.6& $0$ & $1$ & $-1$ & $-1$ & $1$ & $-6$ & $1$ & $-1$ & $1$ & $5$ & $3$ & $-3$ & $5$ & $13$ & $2$ &\ & $2$& $z$\\
&B.11& $0$ & $1$ & $-3$ & $-1$ & $3$ & $-8$ & $1$ & $-3$ & $1$ & $7$ & $3$ & $-9$ & $5$ & $19$ & $2$ &\ &$2$& $x$\\
\hline
5.11*&B.12& $0$ & $1$ & $-3$ & $-1$ & $3$ & $-2$ & $1$ & $-3$ & $1$ & $1$ & $3$ & $-9$ & $-1$ & $7$ & $-4$ &\ &$2$& $x$\\
\hline
&B.3& $0$ & $1$ & $-1$ & $-1$ & $1$ & $-2$ & $1$ & $-1$ & $1$ & $1$ & $3$ & $-3$ & $1$ & $5$ & $-2$ &\ &$2$& $z$ \\
5.12*&B.4& $0$ & $1$ & $-1$ & $-1$ & $1$ & $-3$ & $1$ & $-1$ & $1$ & $2$ & $3$ & $-3$ & $2$ & $7$ & $-1$ &\ &$2$& $z$ \\
&B.13& $0$ & $1$ & $-3$ & $-1$ & $3$ & $-5$ & $1$ & $-3$ & $1$ & $4$ & $3$ & $-9$ & $2$ & $13$ & $-1$ &\ &$2$& $x$\\
\hline
\multicolumn{20}{|c|}{All $(3,1)$ coplanarities leave the other two points at the same side}\\
\hline
3.6&C.1& $0$ & $1$ & $2$ & $-1$ & $-2$ & $0$ & $1$ & $2$ & $-1$ & $1$ & $3$ & $6$ & $0$ & $0$ & $3$ &\ & $2$&$x$\\
\hline
3.11&C.2& $0$ & $1$ & $2$ & $-1$ & $-2$ & $0$ & $1$ & $2$ & $0$ & $0$ & $3$ & $6$ & $1$ & $-1$ & $1$ &\ & $2$& $x$\\
&C.3& $0$ & $3$ & $6$ & $-3$ & $-6$ & $0$ & $3$ & $6$ & $0$ & $0$ & $9$ & $18$ & $3$ & $-3$ & $3$ &\ & $3$& $x$\\
\hline
5.4*&C.4&$0$ & $1$ & $5$ & $-1$ & $-5$ & $1$ & $1$ & $5$ & $-2$ & $1$ & $3$ & $15$ & $1$ & $-4$ & $7$ &\ & $2$& $y$\\
&C.5& $0$ & $1$ & $7$ & $-1$ & $-7$ & $1$ & $1$ & $7$ & $-2$ & $1$ & $3$ & $21$ & $3$ & $-6$ & $9$ &\ & $2$& $y$\\
\hline
5.6*&C.6& $0$ &$1$ & $3$ & $-1$ & $-3$ & $-2$ & $1$ & $3$ & $1$ & $1$ & $3$ & $9$ & $5$ & $1$ & $2$  &\ & $2$& $x$\\ 
\hline
\hline
\multicolumn{20}{|c|}{ \textbf{Polytopes containing a $(2,2)$ coplanarity, but none of the above}}\\
\hline
\multicolumn{20}{|c|}{One $(2,2)$ coplanarity leaves the other two points at opposite sides}\\
\hline
5.13&D.1& $0$ & $1$ & $-1$ & $1$ & $-1$ & $-4$ & $-1$ & $1$ & $3$ & $-1$ & $-1$ & $1$ & $5$ & $1$ & $-2$ &\ & $2$& $z$\\
&D.2& $0$ & $1$ & $-1$ & $1$ & $-1$ & $-5$ & $-1$ & $1$ & $4$ & $-1$ & $-1$ & $1$ & $7$ & $2$ & $-3$ &\ & $2$& $z$\\
\hline
\multicolumn{20}{|c|}{All $(2,2)$ coplanarities leave the other two points at the same side}\\
\hline
5.5&E.1& $0$ & $1$ & $5$ & $1$ & $5$ & $1$ & $-1$ & $-5$ & $-2$ & $-1$ & $-1$ & $-5$ & $1$ & $2$ & $-3$ &\ & $2$& $y$\\
&E.2& $0$ & $1$ & $7$ & $1$ & $7$ & $2$ & $-1$ & $-7$ & $-3$ & $-1$ & $-1$ & $-7$ & $1$ & $3$ & $-4$ &\ & $2$& $x-y$\\
\hline
\hline
\multicolumn{20}{|c|}{ \textbf{Polytopes containing a $(2,1)$ coplanarity, but none of the above}}\\
\hline
&F.1& $1$&$ -1 $&$ -2 $&$ 1 $&$ 2 $&$ 0 $&$ -1 $&$ -2 $&$ 0 $&$ 0 $&$ -4 $&$ -7 $&$ -1 $&$ 1 $&$ -1$ &\ & $2$& $y$\\
&F.2& $1$&$ -2 $&$ -4 $&$ 1 $&$ 2 $&$ 0 $&$ -1 $&$ -2 $&$ 0 $&$ 0 $&$ -5 $&$ -9 $&$ -2 $&$ 1 $&$ -1$ &\ & $2$& $z$\\
4.21&F.3& $2$&$ -1 $&$ -2 $&$ 1 $&$ 2 $&$ 0 $&$ -1 $&$ -2 $&$ 0 $&$ 0 $&$ -5 $&$ -8 $&$ -1 $&$ 1 $&$ -1$ &\ & $2$& $x-z$\\
&F.4& $1$&$ -3 $&$ -6 $&$ 2 $&$ 4 $&$ 0 $&$ -1 $&$ -2 $&$ 0 $&$ 0 $&$ -7 $&$ -13$&$ -3 $&$ 2 $&$ -1$&\  & $2$& $z$\\
&F.5& $3$&$ -2 $&$ -4 $&$ 1 $&$ 2 $&$ 0 $&$ -1 $&$ -2 $&$ 0 $&$ 0 $&$ -7 $&$ -11$&$ -2 $&$ 1 $&$ -1$&\  & $2$& $x-z$\\
&F.6& $5$&$ -3 $&$ -6 $&$ 2 $&$ 4 $&$ 0 $&$ -1 $&$ -2 $&$ 0 $&$ 0 $&$ -11 $&$ -17$&$ -3 $&$ 2 $&$ -1$&\  & $2$& $x-z$\\
\hline
&F.7& $1$&$ -1 $&$ 1 $&$ 1 $&$ -1 $&$ 0 $&$ -1 $&$ 1 $&$ 0 $&$ 0 $&$ -4 $&$ 2 $&$ 2 $&$ -2 $&$ 2 $&\ & $2$& $y$\\
&F.8& $1$&$ -2 $&$ 2 $&$ 1 $&$ -1 $&$ 0 $&$ -1 $&$ 1 $&$ 0 $&$ 0 $&$ -5 $&$ 3 $&$ 4 $&$ -2 $&$ 2 $&\ & $2$& $z$\\
4.22&F.9& $2 $&$ -1 $&$ 1 $&$ 1 $&$ -1 $&$ 0 $&$ -1 $&$ 1 $&$ 0 $&$ 0 $&$ -5 $&$ 1 $&$ 2 $&$ -2 $&$ 2 $&\ & $2$& $z$\\
&F.10&$ 1$&$ -3 $&$ 3 $&$ 2 $&$ -2 $&$ 0 $&$ -1 $&$ 1 $&$ 0 $&$ 0 $&$ -7 $&$ 5 $&$ 6 $&$ -4 $&$ 2$ &\ & $2$& $z$\\
&F.11&$ 3$&$ -2 $&$ 2 $&$ 1 $&$ -1 $&$ 0 $&$ -1 $&$ 1 $&$ 0 $&$ 0 $&$ -7 $&$ 1 $&$ 4 $&$ -2 $&$ 2$ &\ & $2$& $z$\\
&F.12&$ 5$&$ -3 $&$ 3 $&$ 2 $&$ -2 $&$ 0 $&$ -1 $&$ 1 $&$ 0 $&$ 0 $&$ -11 $&$ 1$&$ 6 $&$ -4 $&$ 2$ &\ & $2$& $z$\\
\hline
&F.13& $1 $&$ -1 $&$ -3 $&$ 1 $&$ 2 $&$ 1 $&$ -1 $&$ -2 $&$ -1 $&$ 0 $&$ -4 $&$ -8 $&$ -4 $&$ 0 $&$ 0$ &\ & $2$& $y$\\
&F.14&$ 1$&$ -1 $&$ -3 $&$ 2 $&$ 4 $&$ 2 $&$ -1 $&$ -2 $&$ -1 $&$ 0 $&$ -5 $&$ -10 $&$ -5 $&$ 0 $&$ 0$&\ & $2$& $z$\\
4.11&F.15&$ 2$&$ -1 $&$ -4 $&$ 1 $&$ 2 $&$ 1 $&$ -1 $&$ -2 $&$ -1 $&$ 0 $&$ -5 $&$ -10 $&$ -5 $&$ 0 $&$ 0$&\  & $2$& $x-z$\\
&F.16&$ 2$&$ -1 $&$ -4 $&$ 3 $&$ 6 $&$ 3 $&$ -1 $&$ -2 $&$ -1 $&$ 0 $&$ -7 $&$ -14 $&$ -7 $&$ 0 $&$ 0$&\  & $2$& $x-z$\\
&F.17&$ 3$&$ -1 $&$ -5 $&$ 2 $&$ 4 $&$ 2 $&$ -1 $&$ -2 $&$ -1 $&$ 0 $&$ -7 $&$ -14 $&$ -7 $&$ 0 $&$ 0$&\ & $2$& $x-z$\\
\hline
\end{tabular}
\caption{Lattice $3$-polytopes of size $6$ and width $>1$ with some coplanarity. The ones that are dps are marked with an * in the first column}
\label{table:6points-I}

\end{table}

\begin{table}[htbp]
\centering
\small
\begin{tabular}{|c|l|rrrrrrrrrrrrrrr|ccc|} 
\hline
\textbf{OM} &\textbf{Id.} &\multicolumn{15}{c|}{\textbf{Volume vector}} &\multicolumn{3}{c|}{\torre{ \textbf{Width,}}{ \textbf{functional}}} \\
\hline
\multicolumn{20}{|c|}{\textbf{Polytopes with no coplanarities and $1$ interior point}}\\
\hline
&G.1 & $1 $&$ -1 $&$ -1 $&$ 1 $&$ 3 $&$ -2 $&$ -1 $&$ -2 $&$ 1 $&$ 1 $&$ -4 $&$ -7 $&$ 3 $&$ 5 $&$ -1$ &\ & $2$& $y$ \\
&G.2 & $1 $&$ -1 $&$ -3 $&$ 1 $&$ 5 $&$ -2 $&$ -1 $&$ -4 $&$ 1 $&$ 1 $&$ -4 $&$ -13 $&$ 1 $&$ 7 $&$ -3$ &\ & $2$& $y$ \\
&G.3 & $1 $&$ -1 $&$ -1 $&$ 1 $&$ 2 $&$ -1 $&$ -2 $&$ -3 $&$ 1 $&$ 1 $&$ -5 $&$ -7 $&$ 2 $&$ 3 $&$ -1$ &\ & $2$& $z$\\
&G.4 & $1 $&$ -1 $&$ -2 $&$ 1 $&$ 3 $&$ -1 $&$ -2 $&$ -5 $&$ 1 $&$ 1 $&$ -5 $&$ -11 $&$ 1 $&$ 4 $&$ -3$ &\ & $2$& $z$\\
&G.5 & $1 $&$ -2 $&$ -5 $&$ 1 $&$ 4 $&$ -3 $&$ -1 $&$ -3 $&$ 1 $&$ 1 $&$ -5 $&$ -13 $&$ 1 $&$ 7 $&$ -2$ &\ & $2$& $x$\\ 
&G.6 & $2 $&$ -1 $&$ -1 $&$ 1 $&$ 5 $&$ -2 $&$ -1 $&$ -3 $&$ 1 $&$ 1 $&$ -5 $&$ -11 $&$ 3 $&$ 7 $&$ -2$ &\ & $2$& $x-z$\\
&G.7 & $2 $&$ -1 $&$ -3 $&$ 1 $&$ 7 $&$ -2 $&$ -1 $&$ -5 $&$ 1 $&$ 1 $&$ -5 $&$ -17 $&$ 1 $&$ 9 $&$ -4$ &\ & $2$& $x-z$\\
&G.8 & $1 $&$ -2 $&$ -3 $&$ 1 $&$ 2 $&$ -1 $&$ -3 $&$ -5 $&$ 1 $&$ 1 $&$ -7 $&$ -11 $&$ 1 $&$ 3 $&$ -2$ &\ & $2$& $z$\\
&G.9 & $2 $&$ -1 $&$ -1 $&$ 1 $&$ 3 $&$ -1 $&$ -3 $&$ -5 $&$ 1 $&$ 2 $&$ -7 $&$ -11 $&$ 2 $&$ 5 $&$ -1$ &\ & $2$& $z$\\
6.2*&G.10 &$2 $&$ -3 $&$ -7 $&$ 1 $&$ 5 $&$ -4 $&$ -1 $&$ -3 $&$ 1 $&$ 1 $&$ -7 $&$ -17$&$ 1$&$ 9$&$ -2$ &\ & $2$& $z$\\
&G.11&$3 $&$ -2 $&$ -1 $&$ 1 $&$ 5 $&$ -3 $&$ -1 $&$ -2 $&$ 1 $&$ 1 $&$ -7 $&$ -11 $&$ 5 $&$ 8 $&$ -1$ &\ & $2$& $z$\\
&G.12&$3 $&$ -1 $&$ -1 $&$ 1 $&$ 4 $&$ -1 $&$ -2 $&$ -5 $&$ 1 $&$ 1 $&$ -7 $&$ -13 $&$ 2 $&$ 5 $&$ -3$ &\ & $2$& $x-z$\\
&G.13&$3 $&$ -1 $&$ -2 $&$ 1 $&$ 5 $&$ -1 $&$ -2 $&$ -7 $&$ 1 $&$ 1 $&$ -7 $&$ -17 $&$ 1 $&$ 6 $&$ -5$ &\ & $2$& $x-z$\\
&G.14&$3 $&$ -2 $&$ -5 $&$ 1 $&$ 7 $&$ -3 $&$ -1 $&$ -4 $&$ 1 $&$ 1 $&$ -7 $&$ -19 $&$ 1 $&$ 10 $&$ -3$ &\ & $2$& $x-z$\\
&G.15&$5 $&$ -2 $&$ -1 $&$ 1 $&$ 3 $&$ -1 $&$ -3 $&$ -4 $&$ 1 $&$ 1 $&$ -11 $&$ -13 $&$ 3 $&$ 4 $&$ -1$ &\ & $2$& $z$\\
&G.16&$5 $&$ -2 $&$ -3 $&$ 1 $&$ 4 $&$ -1 $&$ -3 $&$ -7 $&$ 1 $&$ 1 $&$ -11 $&$ -19 $&$ 1 $&$ 5 $&$ -4$ &\ & $2$& $x-z$\\
&G.17&$5 $&$ -3 $&$ -5 $&$ 1 $&$ 5 $&$ -2 $&$ -2 $&$ -5 $&$ 1 $&$ 1 $&$ -11 $&$ -20 $&$ 1 $&$ 7 $&$ -3$ &\ & $2$& $x-z$\\
&G.18&$3 $&$ -4 $&$ -5 $&$ 1 $&$ 2 $&$ -1 $&$ -5 $&$ -7 $&$ 1 $&$ 1 $&$ -13 $&$ -17 $&$ 1 $&$ 3 $&$ -2$ &\ & $2$& $z$\\
&G.19&$4 $&$ -5 $&$ -7 $&$ 1 $&$ 3 $&$ -2 $&$ -3 $&$ -5 $&$ 1 $&$ 1 $&$ -13 $&$ -19 $&$ 1 $&$ 5 $&$ -2$ &\ & $2$& $z$\\
&G.20&$5 $&$ -3 $&$ -4 $&$ 1 $&$ 3 $&$ -1 $&$ -4 $&$ -7 $&$ 1 $&$ 1 $&$ -13 $&$ -19 $&$ 1 $&$ 4 $&$ -3$ &\ & $2$& $x-z$ \\
\hline
\hline
\multicolumn{20}{|c|}{ \textbf{Polytopes with no coplanarities and $2$ interior points}}\\
\hline
&H.1 &$1 $&$ -1 $&$ 5 $&$ 1 $&$ 1 $&$ -6 $&$ -1 $&$ -2 $&$ 7 $&$ -1 $&$ -4 $&$ 1 $&$ 19 $&$ 5 $&$ -9 $&\ & $2$& $x-2y$ \\
&H.2 &$1 $&$ -1 $&$ 7 $&$ 1 $&$ 1 $&$ -8 $&$ -1 $&$ -2 $&$ 9 $&$ -1 $&$ -4 $&$ 3 $&$ 25 $&$ 7 $&$-11 $ &\ & $2$& $x-z$ \\
&H.3 &$1 $&$ -2 $&$ 5 $&$ 1 $&$ 1 $&$ -7 $&$ -1 $&$ -2 $&$ 9 $&$ -1 $&$ -5 $&$ 1 $&$ 23 $&$ 6 $&$-11 $ &\ & $2$& $x-y$\\
&H.4 &$1 $&$ -2 $&$ 7 $&$ 1 $&$ 1 $&$ -9 $&$ -1 $&$ -2 $&$ 11 $&$ -1 $&$ -5 $&$ 3$&$ 29 $&$ 8 $&$ -13 $ &\ & $2$& $x-z$\\
&H.5 &$2 $&$ -1 $&$ 7 $&$ 1 $&$ 1 $&$ -4 $&$ -1 $&$ -3 $&$ 5 $&$ -1 $&$ -5 $&$ 1 $&$ 17 $&$ 3 $&$ -8 $&\ & $2$& $x-z$\\ 
6.1*&H.6 &$2 $&$ -1 $&$ 11 $&$ 1 $&$ 1 $&$ -6$&$ -1 $&$ -3 $&$ 7 $&$ -1 $&$ -5 $&$ 5$&$ 25 $&$ 5 $&$ -10 $&\ &$2$& $x-z$\\
&H.7 &$2 $&$ -3 $&$ 7 $&$ 1 $&$ 1 $&$ -5 $&$ -1 $&$ -3 $&$ 8 $&$ -1 $&$ -7 $&$ 1 $&$ 23 $&$ 4 $&$-11 $ &\ & $2$& $x-z$\\
&H.8 &$2 $&$ -3 $&$ 11 $&$ 1 $&$ 1 $&$ -7$&$ -1 $&$ -3 $&$ 10 $&$ -1 $&$ -7 $&$ 5$&$ 31 $&$ 6 $&$ -13$ &\ & $2$& $x-z$\\
&H.9 &$3 $&$ -1 $&$ 7 $&$ 2 $&$ 1 $&$ -5 $&$ -1 $&$ -2 $&$ 3 $&$ -1 $&$ -7 $&$ 1 $&$ 16 $&$ 3 $&$ -5 $&\ & $2$& $x$\\
&H.10&$3 $&$ -2 $&$ 13 $&$ 1 $&$ 1 $&$ -5$&$ -1 $&$ -4$&$ 7 $&$ -1 $&$ -7 $&$ 5$&$ 27 $&$ 4 $&$ -11$ &\ & $2$& $x-z$\\
&H.11&$3 $&$ -5 $&$ 11 $&$ 2 $&$ 1 $&$ -9$&$ -1 $&$ -2$&$ 7 $&$ -1 $&$ -11$&$ 5$&$ 32 $&$ 7 $&$ -9$ &\ & $2$& $x-z$\\
&H.12&$5 $&$ -2 $&$ 11 $&$ 3 $&$ 1 $&$ -7$&$ -1 $&$ -2$&$ 3 $&$ -1 $&$ -11$&$ 3$&$ 23 $&$ 4 $&$ -5$ &\ & $3$& $x-z$\\
\hline
\end{tabular}
\caption{Lattice $3$-polytopes of size $6$ and width $>1$ with no coplanarities. All dps}
\label{table:6points-II}
\end{table}

%\begin{table}[htbp]
%\centering
\begin{multicols}{3}
\footnotesize

\[A.1\left( \begin{array}{cccccc}
 0 & 1 & 0 &-1 & 0 & 1\\ 
 0 & 0 & 1 & 0 & 2 & 1\\ 
 0 & 0 & 0 & 0 & 0 & 2\end{array} \right)\]

\[A.2\left( \begin{array}{cccccc}
 0 & 1 & 0 &-1 & 0 & 1\\ 
 0 & 1 & 1 & 1 & 2 & 0\\ 
 0 & 0 & 0 & 0 & 0 & 2\end{array} \right)\]

\[B.1\left( \begin{array}{cccccc}
 0 & 1 & 0 &-1 & 0 & 1\\ 
 0 & 0 & 1 &-1 & 0 & 4\\ 
 0 & 0 & 0 & 0 & 1 & -1\end{array} \right)\]

\[B.2\left( \begin{array}{cccccc}
 0 & 1 & 0 &-1 & 0 & 0\\ 
 0 & 0 & 1 &-1 & 0 & 3\\ 
 0 & 0 & 0 & 0 & 1 & -1\end{array} \right)\]

\[B.3\left( \begin{array}{cccccc}
 0 & 1 & 0 &-1 & 0 & 1\\ 
 0 & 0 & 1 &-1 & 0 & 2\\ 
 0 & 0 & 0 & 0 & 1 & -1\end{array} \right)\]

\[B.4\left( \begin{array}{cccccc}
  0 & 1 & 0 &-1 & 0 & 1\\ 
  0 & 0 & 1 &-1 & 0 & 3\\ 
  0 & 0 & 0 & 0 & 1 & -1\end{array} \right)\]

\[B.5\left( \begin{array}{cccccc}
  0 & 1 & 0 &-1 & 0 & 1\\ 
  0 & 0 & 1 &-1 & 0 & 5\\ 
  0 & 0 & 0 & 0 & 1 & -1\end{array} \right)\]

\[B.6\left( \begin{array}{cccccc}
  0 & 1 & 0 &-1 & 0 & 1\\ 
  0 & 0 & 1 &-1 & 0 & 6\\ 
  0 & 0 & 0 & 0 & 1 & -1\end{array} \right)\]

\[B.7\left( \begin{array}{cccccc}
  0 & 1 & 0 &-1 & 0 & 0\\ 
  0 & 0 & 1 &-1 & 0 & 2\\ 
  0 & 0 & 0 & 0 & 1 & -1\end{array} \right)\]

\[B.8\left( \begin{array}{cccccc}
  0 & 1 & 0 &-1 & 0 & 0\\ 
  0 & 0 & 1 &-1 & 0 & 1\\ 
  0 & 0 & 0 & 0 & 1 & -1\end{array} \right)\]

\[B.9\left( \begin{array}{cccccc}
  0 & 1 & 0 &-1 & 0 & 1\\ 
  0 & 0 & 1 &-1 & 0 & 1\\ 
  0 & 0 & 0 & 0 & 1 & -1\end{array} \right)\]

\[B.10\left( \begin{array}{cccccc}
  0 & 1 & 0 &-1 & 0 & 0\\ 
  0 & 0 & 1 &-1 & 0 & 0\\ 
  0 & 0 & 0 & 0 & 1 & -1\end{array} \right)\]

\[B.11\left( \begin{array}{cccccc}
  0 & 1 & 0 &-1 & 0 & 1\\ 
  0 & 0 & 1 &-1 & 0 & 8\\ 
  0 & 0 & 0 & 0 & 1 & -3\end{array} \right)\]

\[B.12\left( \begin{array}{cccccc}
  0 & 1 & 0 &-1 & 0 & 1\\ 
  0 & 0 & 1 &-1 & 0 & 2\\ 
  0 & 0 & 0 & 0 & 1 & -3\end{array} \right)\]

\[B.13\left( \begin{array}{cccccc}
  0 & 1 & 0 &-1 & 0 & 1\\ 
  0 & 0 & 1 &-1 & 0 & 5\\ 
  0 & 0 & 0 & 0 & 1 & -3\end{array} \right)\]

\[B.14\left( \begin{array}{cccccc}
  0 & 1 & 0 &-1 & 1 & -1\\ 
  0 & 0 & 1 &-1 & 2 & 1\\ 
  0 & 0 & 0 & 0 & 3 & -3\end{array} \right)\]

\[B.15\left( \begin{array}{cccccc}
  0 & 1 & 0 &-1 & 1 & -1\\ 
  0 & 0 & 1 &-1 & 2 & -2\\ 
  0 & 0 & 0 & 0 & 3 & -3\end{array} \right)\]

\[C.1\left( \begin{array}{cccccc}
  0 & 1 & 0 &-1 & 0 & -1\\ 
  0 & 0 & 1 &-1 & 0 & 0\\ 
  0 & 0 & 0 & 0 & 1 & 2\end{array} \right)\]

\[C.2\left( \begin{array}{cccccc}
  0 & 1 & 0 &-1 & 0 & 0\\ 
  0 & 0 & 1 &-1 & 0 & 0\\ 
  0 & 0 & 0 & 0 & 1 & 2\end{array} \right)\]

\[C.3\left( \begin{array}{cccccc}
  0 & 1 & 0 &-1 & 1 & 2\\ 
  0 & 0 & 1 &-1 & 2 & 4\\ 
  0 & 0 & 0 & 0 & 3 & 6\end{array} \right)\]

\[C.4\left( \begin{array}{cccccc}
  0 & 1 & 0 &-1 & 0 & -2\\ 
  0 & 0 & 1 &-1 & 0 & -1\\ 
  0 & 0 & 0 & 0 & 1 & 5\end{array} \right)\]

\[C.5\left( \begin{array}{cccccc}
  0 & 1 & 0 &-1 & 0 & -2\\ 
  0 & 0 & 1 &-1 & 0 & -1\\ 
  0 & 0 & 0 & 0 & 1 & 7\end{array} \right)\]

\[C.6\left( \begin{array}{cccccc}
  0 & 0 & -1 &1 & 1 & 1\\ 
  0 & 1 & -1 &0 & 1 & 2\\ 
  0 & 0 & 0 & 0 & 1 & 3\end{array} \right)\]

\[D.1\left( \begin{array}{cccccc}
  0 & 1 & 0 & 1 & 0 & 3\\ 
  0 & 0 & 1 & 1 & 0 & 4\\ 
  0 & 0 & 0 & 0 & 1 & -1\end{array} \right)\]

\[D.2\left( \begin{array}{cccccc}
  0 & 1 & 0 & 1 & 0 & 4\\ 
  0 & 0 & 1 & 1 & 0 & 5\\ 
  0 & 0 & 0 & 0 & 1 & -1\end{array} \right)\]

\[E.1\left( \begin{array}{cccccc}
  0 & 1 & 0 & 1 & 0 & -2\\ 
  0 & 0 & 1 & 1 & 0 & -1\\ 
  0 & 0 & 0 & 0 & 1 & 5\end{array} \right)\]

\[E.2\left( \begin{array}{cccccc}
  0 & 1 & 0 & 1 & 0 & -3\\ 
  0 & 0 & 1 & 1 & 0 & -2\\ 
  0 & 0 & 0 & 0 & 1 & 7\end{array} \right)\]

\[F.1\left( \begin{array}{cccccc}
  0 &1  &0  &-2 &1 &2\\ 
  0 &1 &0  &-1  & 0 &0\\ 
  0 &1 &1  &-2  & 0 &0\end{array} \right)\]

\[F.2\left( \begin{array}{cccccc}
  0 &1 &0  &-1  & 1 &2\\ 
  0 &2 &0  &-1  & 0 &0\\ 
  0 &1 &1  &-1  & 0 &0\end{array} \right)\]

\[F.3\left( \begin{array}{cccccc}
  0 &0 &1  & 1  &-1  &-2\\ 
  0 &0 &0  & 2  &-1  &-2\\ 
  0 &1 &0  & 1  &-1  &-2\end{array} \right)\]

\[F.4\left( \begin{array}{cccccc}
 0 &0 &1  &-1  & 1  & 2\\ 
 0 &0 &3  &-2  & 0  & 0\\ 
 0 &1 &1  &-1  & 0  & 0\end{array} \right)\]

\[F.5\left( \begin{array}{cccccc}
 0 &0 &1  & 1  &-1  &-2\\ 
 0 &0 &0  & 3  &-2  &-4\\ 
 0 &1 &0  & 1  &-1  &-2\end{array} \right)\]

\[F.6\left( \begin{array}{cccccc}
 0 &1 &2  & 0  &-1  &-2\\ 
 0 &0 &5  & 0  &-2  &-4\\ 
 0 &0 &1  & 1  &-1  &-2\end{array} \right)\]

\[F.7\left( \begin{array}{cccccc}
 0 &1 &0  &-2  & 1  &-1\\ 
 0 &1 &0  &-1  & 0  & 0\\ 
 0 &1 &1  &-2  & 0  & 0\end{array} \right)\]

\[F.8\left( \begin{array}{cccccc}
 0 &1 &0  &-1  & 1  &-1\\ 
 0 &2 &0  &-1  & 0  & 0\\ 
 0 &1 &1  &-1  & 0 &0\end{array} \right)\]

\[F.9\left( \begin{array}{cccccc}
 0 &0 &1  & 1  &-1 &1\\ 
 0 &0 &0  & 2  &-1 &1\\ 
 0 &1 &0  & 1  &-1 &1\end{array} \right)\]

\[F.10\left( \begin{array}{cccccc}
 0 &0 &1  &-1  & 1  &-1\\ 
 0 &0 &3  &-2  & 0  & 0\\ 
 0 &1 &1  &-1  & 0  & 0\end{array} \right)\]

\[F.11\left( \begin{array}{cccccc}
 0 &0 &1  & 1  &-1  & 1\\ 
 0 &0 &0  & 3  &-2  & 2\\ 
 0 &1 &0  & 1  &-1  & 1\end{array} \right)\]

\[F.12\left( \begin{array}{cccccc}
 0 &1 &2  & 0  &-1  & 1\\ 
 0 &0 &5  & 0  &-2  & 2\\ 
 0 &0 &1  & 1  &-1  & 1\end{array} \right)\]

\[F.13\left( \begin{array}{cccccc}
 0 &1 &0  &-2  & 1  & 4\\ 
 0 &1 &0  &-1  & 0  & 1\\ 
 0 &1 &1  &-2  & 0  & 2\end{array} \right)\]

\[F.14\left( \begin{array}{cccccc}
 0 &0  &-1  & 1  & 1  & 1\\ 
 0 &0  &-1  & 2  & 0  &-2\\ 
 0 &1  &-1  & 1  & 0  &-1\end{array} \right)\]

\[F.15\left( \begin{array}{cccccc}
 0 &0 &1  & 1  &-1  &-3\\ 
 0 &0 &0  & 2  &-1  &-4\\ 
 0 &1 &0  & 1  &-1  &-3\end{array} \right)\]

\[F.16\left( \begin{array}{cccccc}
 0 &1  &-1  & 0  & 1  & 2\\ 
 0 &0  &-2  & 0  & 3  & 6\\ 
 0 &0  &-1  & 1  & 1  & 1\end{array} \right)\]

\[F.17\left( \begin{array}{cccccc}
 0 &1 &1  & 0  &-1  &-2\\ 
 0 &0 &3  & 0  &-2  &-4\\ 
 0 &0 &1  & 1  &-1  &-3\end{array} \right)\]
 
\[G.1\left( \begin{array}{cccccc}
 0 &1  &-2  & 0  & 1  & 4\\ 
 0 &1  &-1  & 0  & 0  & 1\\ 
 0 &1  &-2  & 1  & 0  & 3\end{array} \right)\]

\[G.2\left( \begin{array}{cccccc}
 0 &1  &-2  & 0  & 1 & 6\\ 
 0 &1  &-1  & 0  & 0 & 1\\ 
 0 &1  &-2  & 1  & 0  & 3\end{array} \right)\]

\[G.3\left( \begin{array}{cccccc}
 0  &-1  & 1  & 0  & 1  & 1\\ 
 0  &-1  & 2  & 0  & 0  & -1\\ 
 0  &-1  & 1  & 1  & 0  & 0\end{array} \right)\]

\[G.4\left( \begin{array}{cccccc}
 0  &-1  & 0  & 1  & 1  & 3\\ 
 0  &-1  & 0  & 2  & 0  & 1\\ 
 0  &-1  & 1  & 1  & 0  & 0\end{array} \right)\]

\[G.5\left( \begin{array}{cccccc}
 0  & 0  & 1  &-1  & 1  & 1\\ 
 0  & 0  & 2  &-1  &0  & -3\\ 
 0  & 1  & 1  &-1  &0  & -2\end{array} \right)\]

\[G.6\left( \begin{array}{cccccc}
 0  & 0  & 1  & 1  &-1  &-3\\ 
 0  & 0  & 2  & 0  &-1  &-5\\ 
 0  & 1  & 1  & 0  &-1  &-4\end{array} \right)\]

\[G.7\left( \begin{array}{cccccc}
 0  & 0  & 1  & 1  &-1  &-5\\ 
 0  & 0  & 2  & 0  &-1  &-7\\ 
 0  & 1  & 1  & 0  &-1  &-6\end{array} \right)\]

\[G.8\left( \begin{array}{cccccc}
 0  &-1  & 0  & 1  & 1  & 2\\ 
 0  &-2  & 0  & 3  & 0  & 1\\ 
 0  &-1  & 1  & 1  & 0  & 0\end{array} \right)\]

\[G.9\left( \begin{array}{cccccc}
 0  &-1  & 0  & 1  & 1  & 2\\ 
 0  &-2  & 0  & 0  & 3  & 5\\ 
 0  &-1  & 1  & 0  & 1  & 1\end{array} \right)\]

\[G.10\left( \begin{array}{cccccc}
 0  & 1  & 0  &-1  & 1  & 2\\ 
 0  & 0  & 0  &-2  & 3  & 7\\ 
 0  & 0  & 1  &-1  & 1  & 1\end{array} \right)\]

\[G.11\left( \begin{array}{cccccc}
 0  & 0  & 1  & 1  &-1  &-2\\ 
 0  & 0  & 0  & 3  &-2  &-1\\ 
 0  & 1  & 0  & 1  &-1  &-1\end{array} \right)\]

\[G.12\left( \begin{array}{cccccc}
 0  & 1  & 1  & 0  &-1  &-3\\ 
 0  & 3  & 0  & 0  &-2  &-5\\ 
 0  & 1  & 0  & 1  &-1  &-2\end{array} \right)\]

\[G.13\left( \begin{array}{cccccc}
 0  & 1  & 1  & 0  &-1  &-4\\ 
 0  & 3  & 0  & 0  &-2  &-7\\ 
 0  & 1  & 0  & 1  &-1  &-3\end{array} \right)\]

\[G.14\left( \begin{array}{cccccc}
 0  & 0  & 1  & 1  &-1  &-4\\ 
 0  & 0  & 0  & 3  &-2  &-5\\ 
 0  & 1  & 0  & 1  &-1  &-3\end{array} \right)\]

\[G.15\left( \begin{array}{cccccc}
 0  & 0  & 1  & 2  &-1  &-1\\ 
 0  & 0  & 0  & 5  &-2  &-1\\ 
 0  & 1  & 0  & 1  &-1  &-1\end{array} \right)\]

\[G.16\left( \begin{array}{cccccc}
 0  & 0  & 1  & 2  &-1  &-2\\ 
 0  & 0  & 0  & 5  &-2  &-3\\ 
 0  & 1  & 0  & 1  &-1  &-2\end{array} \right)\]

\[G.17\left( \begin{array}{cccccc}
 0  & 2  & 1  & 0  &-1  &-3\\ 
 0  & 5  & 0  & 0  &-2  &-5\\ 
 0  & 1  & 0  & 1  &-1  &-2\end{array} \right)\]

\[G.18\left( \begin{array}{cccccc}
 0  &-1  & 2  & 0  & 1  & 1\\ 
 0  &-1  & 5  & 0  & 0  & -1\\ 
 0  &-1  & 1  & 1  & 0  & 0\end{array} \right)\]

\[G.19\left( \begin{array}{cccccc}
 0  & 1  & 2  &-1  &0  & -1\\ 
 0  & 0  & 5  &-1  &0  & -2\\ 
 0  & 0  & 1  &-1  & 1  & 1\end{array} \right)\]

\[G.20\left( \begin{array}{cccccc}
 0  & 0  & 2  & 1  &-1  &-2\\ 
 0  & 0  & 5  & 0  &-1  &-3\\ 
 0  & 1  & 1  & 0  &-1  &-2\end{array} \right)\]

\[H.1\left( \begin{array}{cccccc}
 0  & 1  & 0  &-2  & 1 &-12\\
 0  & 1  & 0  &-1  & 0 & -7\\
 0  & 1  & 1  &-2  & 0 &-13\end{array} \right)\]

\[H.2\left( \begin{array}{cccccc}
 0  & 0  & 1  &-2  & 1 &-15\\
 0  & 0  & 1  &-1  & 0 & -8\\
 0  & 1  & 1  &-2  & 0 &-17\end{array} \right)\]

\[H.3\left( \begin{array}{cccccc}
 0  & 1  & 0  &-1  & 1 & -7\\
 0  & 2  & 0  &-1  & 0 & -9\\
 0  & 1  & 1  &-1  & 0 & -8\end{array} \right)\]

\[H.4\left( \begin{array}{cccccc}
 0  & 0  & 1  &-1  & 1 & -8\\
 0  & 0  & 2  &-1  & 0 & -9\\
 0  & 1  & 1  &-1  & 0 &-10\end{array} \right)\]

\[H.5\left( \begin{array}{cccccc}
 0  & 0  & 1  & 1  &-1 &  3\\
 0  & 0  & 0  & 2  &-1 &  7\\
 0  & 1  & 0  & 1  &-1 &  2\end{array} \right)\]

\[H.6\left( \begin{array}{cccccc}
 0  & 1  & 0  & 1  &-1 &  4\\
 0  & 0  & 0  & 2  &-1 & 11\\
 0  & 0  & 1  & 1  &-1 &  5\end{array} \right)\]

\[H.7\left( \begin{array}{cccccc}
 0  & 0  & 1  &-1  & 1 & -4\\
 0  & 0  & 0  &-2  & 3 & -7\\
 0  & 1  & 0  &-1  & 1 & -5\end{array} \right)\]

\[H.8\left( \begin{array}{cccccc}
 0  & 1  & 0  &-1  & 1 & -7\\
 0  & 0  & 0  &-2  & 3 &-11\\
 0  & 0  & 1  &-1  & 1 & -6\end{array} \right)\]

\[H.9\left( \begin{array}{cccccc}
 0  & 0  & 1  & 1  &-1 &  2\\
 0  & 0  & 3  & 0  &-2 & -1\\
 0  & 1  & 1  & 0  &-1 & -1\end{array} \right)\]

\[H.10\left( \begin{array}{cccccc}
 0  & 0  & 1  & 1  &-1 &  4\\
 0  & 0  & 0  & 3  &-2 & 13\\
 0  & 1  & 0  & 1  &-1 &  3\end{array} \right)\]

\[H.11\left( \begin{array}{cccccc}
 0  & 1  & 2  &-1  & 0 & -5\\
 0  & 0  & 5  &-2  & 0 & -9\\
 0  & 0  & 1  &-1  & 1 & -4\end{array} \right)\]

\[H.12\left( \begin{array}{cccccc}
 0  & 1  & 0  & 2  &-1 &  4\\
 0  & 0  & 0  & 5  &-2 & 11\\
 0  & 0  & 1  & 1  &-1 &  2\end{array} \right)\]
\end{multicols}